\documentclass[1p,preprint]{elsarticle}

\usepackage{lineno}
\modulolinenumbers[5]

\usepackage{amsmath}
\usepackage{amsthm}
\usepackage{amsfonts}
\usepackage{color}
\usepackage{amssymb}
\usepackage{hyperref}
\usepackage[all,2cell]{xy}
\UseAllTwocells
\xyoption{v2}
\NoResizing
\UseResizing
\DeclareMathAlphabet{\mathpzc}{OT1}{pzc}{m}{it}
\usepackage{enumitem}  
\theoremstyle{theorem}
\newtheorem{theorem}{Theorem}[section]
\newtheorem*{theorem*}{Theorem}
\newtheorem{corollary}[theorem]{Corollary}
\newtheorem*{corollary*}{Corollary}
\newtheorem{proposition}[theorem]{Proposition}
\newtheorem{lemma}[theorem]{Lemma}
\theoremstyle{remark}
\newtheorem{remark}[theorem]{Remark}
\newtheorem{example}[theorem]{Examples}
\theoremstyle{definition}
\newtheorem{definition}[theorem]{Definition}

\def\gromonad{\mathrm{T}_c}
\newcommand{\comprehensiondoctrine}[1]{#1_c}  
\newcommand{\subobject}[3]{\xymatrix{
#1 \ar@{>->}[r]^{#2} & #3
}}
\newcommand{\pullback}[8]{ \xymatrix@+2pc{ 
#1 \pullbackcorner \ar[r]^{#5} \ar[d]_{#6} & #2 \ar[d]^{#7} \\
#3 \ar[r]_{#8} & #4 
}}
\newcommand{\pullbackcorner}[1][ul]{\save*!/#1+1.2pc/#1:(1,-1)@^{|-}\restore}
\newcommand{\theory}{\mathcal{T}}
\newcommand{\lang}{\mathcal{L}}
\newcommand{\signature}{\mathbf{Sg}}
\def\operatorname#1{\mathop{#1}\nolimits}
\newcommand{\bemph}[1]{\textbf{\emph{#1}}}

\newcommand{\alg}[1]{\operatorname{#1\mbox{-}\mathbf{Alg}}}
\newcommand{\ltil}[1]{\widetilde{#1}}
\newcommand{\wht}[1]{\widehat{#1}}

\newcommand{\Sub}{\mathrm{Sub}}
\newcommand{\ovln}[1]{\overline{#1}}
\newcommand{\angbr}[2]{\langle #1,#2 \rangle} 
\newcommand{\comprl}{\ensuremath{ \{ \hspace{-0.6ex} |}} 
\newcommand{\comprr}{  \ensuremath{ | \hspace{-0.6ex}\}}} 
\newcommand{\comp}[1]{\comprl #1 \comprr}
\newcommand{\freccia}[3]{\xymatrix{#2 \colon #1  \ar[r] &  #3}}
\newcommand{\frecciasopra}[3]{\xymatrix{ #1  \ar[r]^{#2} &  #3}}
\newcommand{\duefreccia}[3]{\xymatrix@C=0.5cm{#2 \colon #1  \ar@{=>}[r] &  #3}}
\newcommand{\modificazione}[3]{\xymatrix@C=0.5cm{#2 \colon #1  \ar@{~>}[r] &  #3}}
\newcommand{\doctrine}[2]{\xymatrix{#2 \colon #1^{\op}  \ar[r] & \infsl }}

\newcommand{\duemorfismo}[6]{\xymatrix{
#1^{op} \ar[rrd]^#2_{}="a" \ar[dd]_{#3^{op}}\\
&& \infsl\\
#5^{op}  \ar[rru]_#6^{}="b"
\ar_#4  "a";"b"}}
\newcommand{\comsquare}[8]{ \xymatrix@+1pc{ 
#1 \ar[r]^{#5} \ar[d]_{#6} & #2 \ar[d]^{#7} \\
#3 \ar[r]_{#8} & #4 
}}
\newcommand{\comsquarelargo}[8]{ \xymatrix@+1pc{ 
#1 \ar[rr]^{#5} \ar[d]_{#6} && #2 \ar[d]^{#7} \\
#3 \ar[rr]_{#8} && #4 
}}
\newcommand{\RFinverso}[1]{R_{\langle F(\pr_1),F(\pr_2)\rangle^{-1}}(#1)}
\newcommand{\RFnormale}[1]{R_{\langle F(\pr_1),F(\pr_2)\rangle}(#1)}
\newcommand{\RGinverso}[1]{R_{\langle G(\pr_1),G(\pr_2)\rangle^{-1}}(#1)}

\newcommand{\qot}[2]{#1_{/#2}}

\newcommand{\formulaincontext}[2]{[#2] \; | \; #1}
\newcommand{\termincontext}[3]{[#3] \; | \; #1 : #2}
\def\ED{\mathbf{ED}}
\def\PD{\mathbf{ED}}
\def\CompAdj{\mathbf{CE}}

\def\CED{\mathbf{CED}}
\def\CE{\mathbf{CE}}
\def\QED{\mathbf{QED}}
\def\mR{\mathbb{\mathcal{R}}}
\def\mA{\mathcal{A}}

\def\mC{\mathcal{C}}
\def\mX{\mathcal{X}}
\def\mD{\mathcal{D}}

\def\mG{\mathcal{G}}

\def\des{\mathrm{Des}}
\def\mT{\mathrm{T}}
\def\mS{\mathrm{S}}
\def\funD{\mathrm{D}}
\def\funQ{\mathrm{Q}}
\def\funC{\mathrm{C}}

\def\funU{\mathrm{U}}

\def\Einv{\reflectbox{E}}
\def\pr{\mathrm{pr}}

\def\id{\mathrm{id}}

\def\op{\mathrm{op}}

\def\Sub{\operatorname{\hspace{0.01cm} Sub}}

\newcommand{\sequentincontext}[3]{#1\vdash #2 \;[#3]}

\def\:{\colon}
\def\infsl{\operatorname{\hspace{0.01cm} \mathbf{InfSL}}}

\begin{document}
\begin{frontmatter}
\title{An algebraic approach to the completions of elementary doctrines}
\author{Davide Trotta}
\address{Department of computer science, University of Pisa, Italy}
\ead{trottadavide92@gmail.com}

\begin{abstract}

We provide a thorough algebraic analysis of three known completions having a central role in the exact completions of Lawvere’s doctrines: the one adding comprehensive diagonals (i.e. forcing equality on terms to coincide with the equality predicate), the one adding full comprehensions and the one adding quotients.  We show that all these 2-adjunctions are 2-monadic and that the 2-monads arising from these adjunctions are all property-like. This entails that comprehensive diagonals, full comprehensions and quotients are algebraic properties of an elementary doctrine. Finally, we discuss and present the distributive laws between these 2-monads.
\end{abstract}

\end{frontmatter}

\linenumbers
\section{Introduction}
The topic of completing a given structure with quotients to get a richer one has been widely employed in categorical logic to obtain relative consistency results and its categorical aspects have been studied extensively. The calculus of Partial
Equivalence Relations has many applications in the semantics of programming languages. In Type Theory, models of abstract quotients, known as setoid models, are very useful to formalize mathematical proofs.

Over the years, several constructions and notions of completing a category to an exact category have been introduced both in category theory and categorical logic.

Freyd introduced the notion of exact completion of a regular category in \cite{CA}, Carboni presented the exact completion of a lex category \cite{REC,SFEC,FECLEO} and in recent works \cite{QCFF,EQC,UEC,TECH}, Maietti and Rosolini began to study a categorical structure involved with quotient completions, relativizing the basic concept to a doctrine equipped with a logical structure sufficient to describe the notion of an equivalence relation and quotient.

To this purpose, they considered a generalization of the notion of Lawvere’s hyperdoctrine \cite{AF,DACCC,EHCSAF}, namely elementary doctrine, and they extend the notion of exact completion to elementary doctrines. 

In \cite{EQC,TECH} they proved that this exact completion of elementary doctrines can be obtained as the composite of four minor constructions: the comprehension completion, the quotient completion, the extensional collapse of an elementary doctrine, and finally the last step is the construction of the category of entire functional relations. See \cite[Thm. 4.7]{TECH}.

As pointed out in \cite{TECH} the last construction is a reformulation in the language of doctrines of that introduced by Kelly in \cite{NRRFS}.

In this work, we focus the attention on the first three free constructions involved in the exact completion of a doctrine, since they are those adding new logical structures to a given elementary doctrine.

Our main purpose is to provide a complete algebraic account, employing well-known instruments from the formal theory of monads \cite{TDMT,PDLAVB,OPLS,PDLAVB,UCTSBSSL}, to the three constructions previously mentioned: the comprehension completion, the quotient completion and the comprehensive diagonal completion.

Notice that this approach to completions of doctrines has been applied recently for the case of the existential completion of an elementary doctrine \cite{ECRT} and the elementary completion of a primary doctrine \cite{EMMENEGGER2020106445}. Our work carries on this line and it is part of a more long-term goal whose purpose is to develop a complete description of the main logical constructions and structures in terms of 2-monads and algebras for 2-monads, in order to study logic using formal category theory and universal algebra.

For example, one of the main advantages of these methods is that, using the theory of 2-monads, we can formally distinguish \emph{properties} from \emph{structures} of doctrines and understand how such properties (or structures) can be combined in terms of distributive laws.

Recall that 2-monads can express uniformly and elegantly many algebraic structures, and, in particular, that an action of a 2-monad on a given object encodes a \emph{structure} on that object. When the structure is uniquely determined to within unique isomorphism, to give an object with such a structure is just to give an object with a certain \emph{property}. Those 2-monads for which the algebra structure is essentially unique, if it exists, are called \emph{property-like} \cite{OPLS}.

Therefore, we start by giving a detailed description of the 2-functors and the 2-adjunctions obtained from these completions, and we start our analysis of the 2-monads 
\[\freccia{\ED}{\mT_c,\mT_d,\mT_q}{\ED}\]
where $\ED$ denotes the 2-category of elementary doctrines, and the 2-monads are, respectively, the 2-monad $\mT_c$ of comprehension completion, the 2-monad $\mT_d$ of comprehensive diagonal completion, and finally the 2-monad $\mT$ of quotient completion. 
Then, we study the 2-monadicity of the previous 2-adjunctions. In particular, we prove that the following equivalences of 2-categories hold

\[\CE \equiv \alg{\mT_c}\]
\[\CED \equiv \alg{\mT_d}\]
\[\QED \equiv \alg{\mT_q}\]

where $\CE$ is the 2-category of elementary doctrines with full comprehensions, $\CED$ is the 2-category of elementary doctrines with comprehensive diagonals, and $\QED$ is the 2-category of elementary doctrines with stable quotients.

Moreover, we show that $\mT_c$  is colax-idempotent, $\mT_d$ is pseudo-idempotent and that $\mT_q$ is lax-idempotent. In particular, this implies that all these 2-monads are property-like, and then we can conclude that having comprehensions, quotients or comprehensive diagonals is a \emph{property} of a doctrine, and not only a structure.

Finally, we conclude by showing that the 2-monad $\mT_q$ can be lifted to a 2-monad on the 2-category $\alg{\mT_c}$, and hence that there exists a distributive law of 2-monads $\freccia{\mT_c\mT_q}{\delta}{\mT_q\mT_c}$, while 2-monad $\mT_d$ cannot be lifted either on $\alg{\mT_q}$ or $\alg{\mT_c}$.

%
%
%
In the sections \ref{section 2dim monad} and \ref{sec doctrine} we recall definitions and results on 2-monads and doctrines as needed for the rest of the paper.

In section \ref{sec doctrine with comp} we recall the notion of doctrine with full comprehensions and we construct the 2-functor and the 2-monad $\mT_c$ coming from comprehension completion, showing that it is colax-idempotent and that $\CE\equiv \alg{\mT_c}$.

In section \ref{section comprensive diag} we recall the notion of doctrine with comprehensive diagonals and we construct the 2-functor and the 2-monad $\mT_d$ coming from comprehensive diagonal completion, and we show that it is pseudo-idempotent and that $\CED\equiv \alg{\mT_d}$.

In section \ref{section doctrine with quot} we recall the notion of doctrine with quotients and we construct the 2-functor and the 2-monad $\mT_d$ coming from quotient completion, showing that it is lax-idempotent and that $\QED\equiv \alg{\mT_q}$.

Finally, in section \ref{section pseudo-distributive laws} we discuss and present the distributive laws between the 2-monads.

\section{Two-dimensional monads}\label{section 2dim monad}

This section is devoted to recall some notions and results regarding the formal theory of monads, and to fix the notation. We mainly follow the usual conventions as in \cite{TDMT,PDLAVB,UCTSBSSL}, and we refer the reader to the works of Kelly and Lack
\cite{OPLS}, Tanaka and Power \cite{PDLAVB,UCTSBSSL}, and for a more
general and complete description of these topics one can see the
Ph.D. thesis of Tanaka \cite{PHDTT}, the articles of Marmolejo
\cite{CPLR,DLP} and the work of Kelly \cite{RE2C}. 

Recall that a \bemph{2-monad} $(\mT,\mu, \eta)$ on a 2-category $\mA$ is a
2-functor $\freccia{\mA}{\mT}{\mA}$ together 2-natural transformations
$\freccia{\mT^2}{\mu}{\mT}$ and $\freccia{1}{\eta}{\mT}$ such that the
following diagrams commute
\[\xymatrix@+1pc{
\mT^3\ar[d]_{\mu\mT} \ar[r]^{\mT\mu} &\mT^2\ar[d]^{\mu}\\
\mT^2 \ar[r]_{\mu} & \mT
}\qquad
\xymatrix@+1pc{
\mT \ar[dr]_{\id}\ar[r]^{\eta \mT}& \mT^{2} \ar[d]_{\mu} & \mT \ar[l]_{\mT \eta}\ar[dl]^{\id}\\
& \mT.
}\]
A $\mT$-\bemph{algebra} is a pair $(A,a)$ where $A$ is an object of $\mA$ and $\freccia{\mT A}{a}{A}$ is a 1-cell such that the diagrams 
\[\xymatrix@+1pc{
\mT^2 A  \ar[r]^{\mT a} \ar[d]_{\mu_A} & \mT A \ar[d]^{a}\\
\mT A \ar[r]_a & A
}\qquad
\xymatrix@+1pc{
A \ar[rd]_{1_A} \ar[r]^{\eta_A}& \mT A \ar[d]^{a}\\
& A
}\]
commute.
A \bemph{strict} $\mT$-\bemph{morphism} from a $\mT$-algebra $(A,a)$ to a $\mT$-algebra $(B,b)$ is a 1-cell $\freccia{A}{f}{B}$ such that the following diagram commutes:
\[\xymatrix@+1pc{
\mT A \ar[d]_{a} \ar[r]^{\mT f} & \mT B \ar[d]^b\\
A \ar[r]_f &B
}\] 
while a \bemph{lax} $\mT$-\bemph{morphism} from a $\mT$-algebra $(A,a)$ to a $\mT$-algebra $(B,b)$ is a pair $(f,\ovln{f})$ where $f$ is a 1-cell $\freccia{A}{f}{B}$ and $\ovln{f}$ is a 2-cell
\[\xymatrix@+1pc{
\mT A \ar[d]_{a} \ar[r]^{\mT f} \xtwocell[r]{}<>{_<5>\ovln{f}}& \mT B \ar[d]^b\\
A \ar[r]_f &B
}\] 
which satisfies the following \bemph{coherence} conditions:
\[\xymatrix@+1pc{
\mT^2A \ar[d]_{\mu_A}\ar[r]^{\mT^2f} &\mT B \ar[d]^{\mu_B}&&\mT^2A \ar[d]_{\mT a}\xtwocell[r]{}<>{_<5>\;\;\;\mT\ovln{f}}\ar[r]^{\mT^2f} &\mT B \ar[d]^{\mT b}\\ 
\mT A \ar[d]_{a} \ar[r]^{\mT f} \xtwocell[r]{}<>{_<5>\ovln{f}}&\mT B \ar[d]^b\ar@{}[rr]|{=}&& \mT A \ar[d]_{a} \ar[r]^{\mT f} \xtwocell[r]{}<>{_<5>\ovln{f}}&\mT B \ar[d]^b \\
A \ar[r]_f &B&& A \ar[r]_f &B
}\]
and
\[\xymatrix@+1pc{
A \ar[d]_{\eta_A}\ar[r]^f &B\ar[d]^{\eta_B}&&
A\ar[dd]_{1_A}\ar[r]^f& B\ar[dd]^{1_B}\\
\mT A \ar[d]_{a} \ar[r]^{\mT f} \xtwocell[r]{}<>{_<5>\ovln{f}}&\mT B \ar[d]^b\ar@{}[rr]|{=}&& \\
A \ar[r]_f & B&& A\ar[r]_f &B.
}\]
Observe that regions in which no 2-cell is written commute, so they
are deemed to contain the identity 2-cell.

A lax morphism $(f,\ovln{f})$ in which $\ovln{f}$ is invertible
is said $\mT$-\bemph{morphism}. Hence, a strict $\mT$-morphism is a
$\mT$-morphism where $\ovln{f}$ is the identity 2-cell.

The category of $\mT$-algebras and lax $\mT$-morphisms becomes a
2-category introducing the
$\mT$-transformations as 2-cells: a $\mT$-\bemph{transformation} from the
1-cell $\freccia{(A,a)}{(f,\ovln{f})}{(B,b)}$ to
$\freccia{(A,a)}{(g,\ovln{g})}{(B,b)}$ is a 2-cell
$\duefreccia{f}{\alpha}{g}$ in $\mA$ which satisfies the following
coherence condition
\[\xymatrix{
\mT A \ar[dd]_{a}\rrtwocell<4>^{\mT f}_{\mT g}{\;\;\;\mT\alpha} \xtwocell[rr]{}<>{_<7>\ovln{g}}& &\mT B\ar[dd]^b&&\mT A \ar[dd]_{a}\ar@/^/[rr]^{\mT f} \xtwocell[rr]{}<>{_<7>\ovln{f}}& &\mT B\ar[dd]^b\\
&&\ar@{}[rr]|{=}&&\\
A \ar@/_/[rr]_g&& B&& A\rrtwocell<4>^f_g{\alpha} && B
}\]
expressing compatibility of $\alpha$ with $\ovln{f}$ and $\ovln{g}$. 

Using the notion of $\mT$-morphism, one can express in precise mathematical terms what it
means that an action of a monad $\mT$ on an object $A$ is \emph{unique
to within a unique isomorphism}. In \cite{OPLS} a $\mT$-algebra
structure is essentially unique if, given two actions
$\freccia{\mT A}{a,a'}{A}$, there is a unique invertible 2-cell
$\duefreccia{a}{\alpha}{a'}$ such that
$\freccia{(A,a)}{(1_A,\alpha)}{(A,a')}$ is a morphism of
$\mT$-algebras. 
This is fixed by the following definition of property-like 2-monad.

A 2-monad $(\mT,\mu, \eta)$ is said \bemph{property-like} if it
satisfies the following conditions: 
\begin{itemize}
\item for every $\mT$-algebras $(A,a)$ and $(B,b)$, and for every invertible 1-cell $\freccia{A}{f}{B}$ there exists a unique invertible 2-cell $\ovln{f}$
\[\xymatrix@+1pc{
\mT A \ar[d]_{a} \ar[r]^{\mT f} \xtwocell[r]{}<>{_<5>\ovln{f}}& \mT B \ar[d]^b\\
A \ar[r]_f &B
}\]
such that $\freccia{(A,a)}{(f,\ovln{f})}{(B,b)}$ is a morphism of $\mT$-algebras;
\item  for every $\mT$-algebras $(A,a)$ and $(B,b)$, and for every 1-cell $\freccia{A}{f}{B}$ if there exists a 2-cell $\ovln{f}$
\[\xymatrix@+1pc{
\mT A \ar[d]_{a} \ar[r]^{\mT f} \xtwocell[r]{}<>{_<5>\ovln{f}}& \mT B \ar[d]^b\\
A \ar[r]_f &B
}\]
such that $\freccia{(A,a)}{(f,\ovln{f})}{(B,b)}$ is a lax morphism of $\mT$-algebras, then it is the unique 2-cell with such property.
\end{itemize}

We say that a 2-monad $(\mT,\mu, \eta)$ is 
\bemph{lax-idempotent} when, for 
every $\mT$-algebras $(A,a)$ and $(B,b)$, and for every 1-cell
$\freccia{A}{f}{B}$, there exists a unique 2-cell $\ovln{f}$ 
\[\xymatrix@+1pc{
\mT A \ar[d]_{a} \ar[r]^{\mT f} \xtwocell[r]{}<>{_<5>\ovln{f}}& \mT B \ar[d]^b\\
A \ar[r]_f &B
}\]
such that $\freccia{(A,a)}{(f,\ovln{f})}{(B,b)}$ is a lax morphism of
$\mT$-algebras. In particular every lax-idempotent 2-monad is property-like.

We conclude this section recalling the notion of distributive law between 2-monads, and we refer to \cite{PHDTT,PDLAVB,UCTSBSSL} for a complete exposition of these notions in the general context of pseudo-monads. 

Since the notion of 2-monad represents an elegant way to
describe a structure on a category, the notion of \emph{distributive laws} express how two or more such structures on a category can be combined.

Since in our work all the monads will be simply 2-monads, to simplify the reading, the results are presented in \emph{strict} version.

In particular, given two 2-monads $(\mS,\mu^S,\eta^S)$
and $(\mT,\mu^T,\eta^T)$ on a 2-category
$\mA$, a \bemph{distributive law}
$\delta$ of
$\mS$ over $\mT$ is a natural transformation $\freccia{\mS\mT}{\delta}{\mT\mS}$
such that the following diagrams commute
\[\xymatrix@+1pc{
\mS^2\mT\ar[d]_{\mu^ST} \ar[r]^{S\delta} & \mS\mT\mS \ar[r]^{\delta \mS}& \mT\mS^2\ar[d]^{\mT\mu^S} & \mS\mT^2\ar[d]_{\mS\mu^T} \ar[r]^{\delta \mT}& \mT\mS\mT \ar[r]^{\mT\delta} & \mT^2\mS\ar[d]^{\mu^T \mS}\\
\mS\mT \ar[rr]_{\delta} && \mT\mS &\mS\mT \ar[rr]_{\delta} && \mT\mS
}\]

\[\xymatrix@+1pc{
\mT \ar[d]_{\eta^S \mT} \ar[rd]^{\mT\eta^S}&& \mS \ar[d]_{\mS\eta^T} \ar[rd]^{\eta^T \mS}\\
\mS\mT \ar[r]_{\delta} &\mT\mS& \mS\mT \ar[r]_{\delta} &\mT\mS.
}\]
By a \bemph{lifting} of a 2-monad $\mT$ to the 2-category $\alg{\mS}$ of $\mS$-algebras we mean a 2-monad $\ltil{\mT}$ on the 2-category $\alg{\mS}$ such that $\funU_S\ltil{\mT}=\mT\funU_S$ where $\funU_S$ is the forgetful 2-functor for the 2-monad $\mS$.
\begin{theorem}\label{theorem ps-dist-law equivalent to lifting}
To give a distributive law $\freccia{\mS\mT}{\delta}{\mT\mS}$ of 2-monad is equivalent to give a lifting of the 2-monad $\mT$ to a 2-monad $\ltil{\mT}$ on $\alg{\mS}$.
\end{theorem}

\begin{theorem}\label{theorem the composite ps-monad is a ps monad}
Given 2-monads $(\mS,\mu^S,\eta^S)$ and $(\mT,\mu^T,\eta^T)$ on a 2-category $\mA$ and a distributive law $\freccia{\mS\mT}{\delta}{\mT\mS}$, the composite 2-functor $\mT\mS$ acquires the structure for a 2-monad on $\mA$, with multiplication given by
\[\xymatrix@+1pc{
\mT\mS\mT\mS \ar[r]^{\mT\delta \mS} & \mT\mT\mS\mS \ar[r]^{\mu^T \mu^S}& \mT\mS
}\]
and $\alg{\mT\mS}$ is canonically isomorphic to $\alg{\ltil{\mT}}$.
\end{theorem}

\section{The notion of elementary doctrine}\label{sec doctrine}
F.W.  Lawvere introduced the notion of hyperdoctrine in a series of seminal papers \cite{AF,DACCC,EHCSAF} to synthesize the structural properties of logical systems. Lawvere’s crucial intuition was to consider logical languages and theories as hyperdoctrines to study their 2-categorical properties.

In recent years, the notion of hyperdoctrine has been both specialized and generalized in several contexts. In this work we use the notion of \emph{elementary doctrine} introduced in \cite{QCFF,EQC,UEC} in order to generalize the completion of a categorical structure with quotients. The main idea was to relativize the concept of quotient completion to a many sorted logic, represented categorically by a doctrine validating the logical structure needed to express the notion of equivalence relations.

For the rest of the section $\mC$ is assumed to be a category with binary products, and we denote by $\infsl$ the category of inf-semilattices, i.e. the objects of $\infsl$ are posets with finite meets, and morphisms are functions between them which preserve finite meets.

An \bemph{elementary doctrine} on the category $\mC$ is an indexed inf-semilattice $\doctrine{\mC}{P}$ such that for every $A$ in $\mC$ there exists an object $\delta_A$ in $P(A\times A)$ such that:
\begin{enumerate}
\item the assignment
\[\Einv_{\angbr{\id_A}{\id_A}}(\alpha):=P_{\pr_1}(\alpha)\wedge \delta_A\]
for $\alpha$ in $PA$ determines a left adjoint to $\freccia{P(A\times A)}{P_{\angbr{\id_A}{\id_A}}}{PA}$;
\item for every morphism $e$ of the form $\freccia{X\times A}{\angbr{\pr_1,\pr_2}{\pr_2}}{X\times A\times A}$ in $\mC$, the assignment
\[ \Einv_{e}(\alpha):= P_{\angbr{\pr_1}{\pr_2}}(\alpha)\wedge P_{\angbr{\pr_2}{\pr_2}}(\delta_A)\]
for $\alpha$ in $P(X\times A)$ determines a left adjoint to $\freccia{P(X\times A \times A)}{P_e}{P(X\times A)}$.
\end{enumerate}
\begin{example}\label{example doctrines}
The following examples of elementary doctrine are discussed in \cite{AF,EQC}.
\begin{enumerate}
\item Let $\mC$ be a category with finite limits. The  functor \[\doctrine{\mC}{{\Sub_{\mC}}}\]
is an elementary doctrine, where $\Sub_{\mC}$ is the functor assigning to an object $A$ of $\mC$ the poset $\Sub_{\mC}(A)$ of subobjects of $A$ and, for an arrow $\frecciasopra{B}{f}{A}$ the morphism $\freccia{\Sub_{\mC}(A)}{\Sub_{\mC}(f)}{\Sub_{\mC}(B)}$ is given by pulling a subobject back along $f$. The fibered equalities are the diagonal arrows. 
\item Consider a category $\mD$ with finite products and weak pullbacks. The elementary doctrine of weak subobjects is given by the functor 
\[\doctrine{\mD}{{\Psi_{\mD}}}\]
where $\Psi_{\mD}(A)$ is the poset reflection of the slice category $\mD/A$, and for an arrow $\frecciasopra{B}{f}{A}$, the homomorphism $\freccia{\Psi_{\mD}(A)}{\Psi_{\mD}(f)}{\Psi_{\mD}(B)}$ is given by a weak pullback of an arrow $\frecciasopra{X}{g}{A}$ with $f$.
\item Let $\theory$ be a theory in a first order language $\lang$. We define a primary doctrine 
\[\doctrine{\mC_{\theory}}{LT}\]
where $\mC_{\theory}$ is the category of lists of variables and term substitutions:
\begin{itemize}
\item \bemph{objects} of $\mC_{\theory}$ are finite lists of variables $\vec{x}:=(x_1,\dots,x_n)$, and we include the empty list $()$;
\item a \bemph{morphisms} from $(x_1,\dots,x_n)$ into $(y_1,\dots,y_m)$ is a substitution $[t_1/y_1,\dots, t_m/y_m]$ where the terms $t_i$ are built in $\signature$ on the variable $x_1,\dots, x_n$;
\item the \bemph{composition} of two morphisms $\freccia{\vec{x}}{[\vec{t}/\vec{y}]}{\vec{y}}$ and $\freccia{\vec{y}}{[\vec{s}/\vec{z}]}{\vec{z}}$ is given by the substitution
\[ \freccia{\vec{x}}{[s_1[\vec{t}/\vec{y}]/z_k,\dots, s_k[\vec{t}/\vec{y}]/z_k]}{\vec{z}}.\]
\end{itemize}
The functor $\doctrine{\mC_{\theory}}{LT}$ sends a list $(x_1,\dots,x_n)$ in the class $LT(x_1,\dots,x_n)$  of all well formed formulas in the context $(x_1,\dots,x_n)$. We say that $\psi\leq \phi$ where $\phi,\psi\in LT(x_1,\dots,x_n)$ if $\psi\vdash_{\theory}\phi$, and then we quotient in the usual way to obtain a partial order on $LT(x_1,\dots,x_n)$. Given a morphism of $\mC_{\theory}$ 
\[\freccia{(x_1,\dots,x_n)}{[t_1/y_1,\dots,t_m/y_m]}{(y_1,\dots,y_m)}\]
the functor $LT_{[\vec{t}/\vec{y}]}$ acts as the substitution
$LT_{[\vec{t}/\vec{y}]}(\psi(y_1,\dots,y_m))=\psi[\vec{t}/\vec{y}]$.

The doctrine $\doctrine{\mC_{\theory}}{LT}$ is elementary exactly when $\theory$ has an equality predicate. For all the detail we refer to \cite{QCFF}, and for the case of a many sorted first order theory we refer to \cite{CLP}.

\end{enumerate}
\end{example}
Elementary doctrines form a 2-category denoted by $\ED$ where
\begin{itemize}
\item \bemph{0-cells} are elementary doctrines;
\item
a \bemph{1-cell} is a pair $(F,b)$
\[\duemorfismo{\mC}{P}{F}{b}{\mD}{R}\]
such that $\freccia{\mC}{F}{\mD}$ is a functor preserving products, and $\freccia{P}{b}{R\circ F^{\op}}$ is a natural transformation preserving the structures. More explicitly, for every object $A$ in $\mC$, the function $b_A$ preserves finite infima and
\[ b_{A\times A}(\delta_A)=R_{\angbr{F\pr_1}{F\pr_2}}(\delta_{FA}).\]
\item a \bemph{2-cell} is a natural transformation $\freccia{F}{\theta}{G}$ such that for every object $A$ in $\mC$ and every element $\alpha$ in the fibre $PA$, we have 
\[b_A(\alpha)\leq R_{\theta_A}(c_A(\alpha))\]
\end{itemize}

\section{Elementary doctrines with comprehensions}\label{sec doctrine with comp}
In \cite{QCFF,EQC,UEC} the authors intend to develop doctrines that may interpret constructive theories for mathematics. They observe that a crucial property an elementary doctrine should verify in order to sustain such interpretation relates to the axiom of comprehension.
\begin{definition}\label{def comrehension}
Let $\doctrine{\mC}{P}$ be a elementary doctrine and let $\alpha$ be an element of $P(A)$. A \bemph{comprehension} of $\alpha$ is an arrow $\freccia{X}{{\comp{\alpha}}}{A}$ of $\mC$ such that $P_{\comp{\alpha}}(\alpha)=\top_X$ and, for every $\freccia{Z}{f}{A}$ such that $P_f(\alpha)=\top_Z$, there exists a unique arrow $\freccia{Z}{g}{X}$ such that $f=\comp{\alpha} \circ g$.
\end{definition}
One says that $P$ \bemph{has comprehensions} if every $\alpha$ has a comprehension, and that $P$ \bemph{has full comprehensions} if, moreover, $\alpha\leq \beta$ in $P(A)$ whenever $\comp{\alpha}$ factors through $\comp{\beta}$.
\begin{example}\label{example subobject doctrine has comprehensions}
Let us consider the sub-objects doctrine $\doctrine{\mC}{\Sub_{\mC}}$ defined in Example \ref{example doctrines}. In this case, for every object $A$ and every $\alpha=[\subobject{B}{\alpha}{A}]$ in $\Sub_{\mC}(A)$, the comprehension $\comp{\alpha}$ is the arrow $\subobject{B}{\alpha}{A}$ in $\mC$. Moreover, the doctrine $\Sub_{\mC}$ has full comprehensions.
\end{example}
The intuition is that a comprehension morphism represents the subset of elements of the object $A$ obtained by comprehension with the predicate $\alpha$. 

In the internal language of a doctrine $\doctrine{\mC}{P}$, a comprehension of a formula $\formulaincontext{\phi (a)}{a:A}$ is a term $\termincontext{\comp{a:A \; |\; \phi(a)}(x)}{A}{x:X}$ such that
\[ \sequentincontext{\top}{\phi(\comp{a:A | \phi(a)}(x))}{x:X}\]
and any other term which this property can be obtained from $\comp{a:A\; |\; \phi(a)}(x)$ by an unique substitution.

\begin{remark}\label{remark compr has pullbacks}
For every $\freccia{A'}{f}{A}$ in $\mC$ then the mediating arrow between the comprehensions $\freccia{X}{\comp{\alpha}}{A}$ and $\freccia{X'}{\comp{P_f(\alpha)}}{A'}$ produces a pullback
\[\pullback{X'}{A'}{X}{A.}{\comp{P_f(\alpha)}}{f'}{f}{\comp{\alpha}}\]
Thus comprehensions are stable under pullbacks.
Moreover it is straightforward to verify that if $\freccia{B}{\comp{\alpha}}{A}$ is a comprehension of $\alpha$, then $\comp{\alpha}$ is monic.

Observe the stability under pullbacks of comprehensions implies that if if $\alpha\leq \beta$, where $\alpha,\beta \in P(A)$, then the unique arrow $a$ such that the following diagram commutes
\[\xymatrix@+2pc{ A_{\alpha} \ar@{-->}[d]_a\ar[rd]^{\comp{\alpha}}\\
A_{\beta} \ar[r]_{\comp{\beta}}& A}\]
is a comprehension. In particular it is the comprehension $a=\comp{P_{\comp{\beta}}(\alpha)}$, because we have that the following is a pullback
\[\pullback{X'}{A'}{X}{A}{\comp{P_{\comp{\alpha}}(\beta)}}{\comp{P_{\comp{\beta}}(\alpha)}}{\comp{\alpha}}{\comp{\beta}}\]
and since $\top_{A'}=P_{\comp{\alpha}}(\alpha)\leq P_{\comp{\alpha}}(\beta)$, we have $\comp{P_{\comp{\alpha}}(\beta)}=\id$ and then $a=\comp{P_{\comp{\beta}}(\alpha)}$.
\end{remark}

As observed in \cite{Rosolini2016RelatingQC}, to view  comprehensions  as  logical  constructors \cite{CLTT} we need to assume that a choice of comprehensions is available in the doctrine.

In details, an elementary doctrine $\doctrine{\mC}{P}$ \bemph{has a choice of comprehensions} if there is a function $\comp{-}$ assigning a comprehension $\freccia{A_{\alpha}}{\comp{\alpha}}{A}$ to every object $\alpha$ of $P(A)$.  Similarly, we say that $P$ \bemph{has a choice of full comprehensions} if $P$ has a choice of comprehensions and these are full.

\medskip
\noindent
\textbf{Notation:} since in the rest of this work we will always use doctrines with a choice of comprehensions, from now on, when we say that an elementary doctrine $P$ has comprehensions, or full comprehensions, we assume that it has a choice of comprehensions, or full comprehensions. 

\medskip
\noindent
\begin{remark}\label{remark fibration and grot. constr}
In many senses it is more general to treat
the abstract theory of the relevant structures for the present paper in terms of fibrations.
In fact, a doctrine $\doctrine{\mC}{P}$ determines a faithful fibration
 \[\freccia{\mG_P}{p_P}{\mC}\]
by a well-known, general construction due to Grothendieck, see \cite{CLTT,QCFF}. We recall very briefly that construction in the present situation. 

The data for the \bemph{total category} $\mG_P$ are:
\begin{itemize}
\item \textbf{an object} is a pair $(A,\alpha)$, where $A$ is in $\mC$ and $\alpha$ is in $P(A)$
\item \textbf{an arrow} $\freccia{(A,\alpha)}{f}{(B,\beta)}$ is an arrow $\freccia{A}{f}{B}$ of $\mC$ such that $\alpha\leq P_f(\beta)$.
\end{itemize}
The projection on the first component extends to a functor $\freccia{\mG_P}{p_P}{\mC}$ which is faithful, with a right inverse right adjoint. Setting up an appropriate 2-category for each structure (one for primary doctrines, one for faithful fibrations as above), it is easy to see that the two constructions extend to an equivalence between those 2-categories. The notions of comprehensions, full comprehensions and the requirement that comprehensions compose can be translated using the previous construction in the language of fibrations. In this case they are called respectively fibration with \emph{subset types}, \emph{full subset types} and \emph{strong coproducts}. In particular the terminology strong coproducts comes from dependent type theory, see \cite[Chapter 10]{CLTT} and \cite{FSF}.
\end{remark}

We denote by $\CE$ the 2-category of elementary doctrines with full comprehensions, and the 1-cells are those $(F,b)$ such that the functor $F$ preserves comprehensions, i.e. $F(\comp{\alpha})=\comp{b_A(\alpha)}$. The 2-cells remain the same.

We recall the construction used in \cite{EQC} to freely add comprehensions to a given elementary doctrine. Given an elementary doctrine $\doctrine{\mC}{P}$
we define the category $\mG_P$ as in Remark \ref{remark fibration and grot. constr}:
\begin{itemize}
\item \textbf{an object} of $\mG_P$ is a pair $(A, \alpha)$, where $A$ is in $\mC$ and $\alpha$ is in $P(A)$;
\item \textbf{a morphism} $\freccia{(A,\alpha)}{f}{(B,\beta)}$ is a morphism $\freccia{A}{f}{B}$ in $\mC$ such that $\alpha\leq P_{f}(\beta)$;
\end{itemize}
The functor $P$ extends to functor $\doctrine{\mG_P}{P_c}$ by setting 
\begin{itemize}
\item $P_c(A,\alpha)=\{  \gamma\, \in \, P(A) \,  | \,  \gamma\leq \alpha\}$;
\item $\freccia{P_c(B,\beta)}{P_c(f)}{P_c(A,\alpha)}$ sends $\gamma \leq \beta$ into $P_f(\gamma)\wedge \alpha$.
\end{itemize}
With these previous assignments, it is direct to check that the functor $\doctrine{\mG_P}{P_c}$ is an elementary doctrine and it has full comprehensions. In particular,
one can observe that for every object $(A,\alpha)$ of $\mG_P$ we can define 
\[ \delta_{(A,\alpha)}:=\delta_A \wedge \alpha\boxtimes \alpha\]
where $\alpha \boxtimes \alpha:= P_{\pr_1}(\alpha)\wedge P_{\pr_2}(\alpha)$, while the comprehension of an element $\alpha\in P_c(A,\beta)$ is given by the arrow $\freccia{(A,\alpha)}{\comp{\alpha}:=\id_A}{(A,\beta)}$.

\noindent
Now we prove that the assignment $P\mapsto P_c$ can be extended to 2-functor 
\[\freccia{\ED}{\funC}{\CE}\]
and we start defining how it acts on the 1-cells and 2-cells in $\ED$.

Therefore, let us consider two elementary doctrines $\doctrine{\mC}{P}$ and $\doctrine{\mD}{R}$, and consider a 1-cell $(F,b)$ of $\ED$:
\[\duemorfismo{\mC}{P}{F}{b}{\mD}{R}\]
We want to prove that the pair $(\wht{F},\wht{b})$ where:
\begin{itemize}
\item $\wht{F}(A,\alpha)$ is $(FA,b_A(\alpha))$ for every $(A,\alpha) \in \mG_P$;
\item $\wht{F}(f)$ is $F(f)$ for every $\freccia{(A,\alpha)}{f}{(B,\beta)}$;
\item $\wht{b}$ is the restriction of $b$ on $P_c$;
\end{itemize}
is a 2-cell in $\CE$:
\[\xymatrix{
\mG_P^{op} \ar[rrd]^{P_c}_{}="a" \ar[dd]_{\wht{F}^{op}}\\
&& \infsl\\
\mG_R^{op}  \ar[rru]_{R_c}^{}="b"
\ar_{\wht{b}}  "a";"b"}\]

\begin{lemma}
$(\wht{F},\wht{b})$ is a 1-cell in $\CE$.
\end{lemma}
\begin{proof}
It is direct to show that $\freccia{\mG_P}{\wht{F}}{\mG_R}$ is a preserving products functor and that $\wht{b}$ is a natural transformation. First we show that $(\wht{F},\wht{b})$ is a 1-cell of elementary doctrines, and then we show that it preserves comprehensions.
Hence, we start observing that
\[(R_c)_{\angbr{F(\pr_1)}{F(\pr_2)}}(\delta_{(FA,b_A(\alpha))})=R_{\angbr{F(\pr_1)}{F(\pr_2)}}(b_A(\alpha)\boxtimes b_A(\alpha)\wedge \delta_{FA})\wedge b_{A\times A}(\alpha \boxtimes \alpha) \]
which is equal to
\[ R_{\angbr{F(\pr_1)}{F(\pr_2)}}(R_{\pr_1'}(b_A(\alpha))\wedge R_{\pr_2'}(b_A(\alpha)))\wedge b_{A\times A}(\delta_A)\wedge b_{A\times A}(\alpha \boxtimes \alpha)     \]
where  $\freccia{FA \times FA}{\pr_i'}{FA}$. Moreover, since  $b$ is a natural transformation, the diagram
\[ \comsquarelargo{PA}{P(A\times A)}{RFA}{RF(A\times A).}{P_{\pr_i}}{b_A}{b_{A\times A}}{R_{F(\pr_i)}}\]
commutes. This implies that 
\[(R_c)_{\angbr{F(\pr_1)}{F(\pr_2)}}(\delta_{(FA,b_A(\alpha))})=b_{A\times A}(P_{\pr_1}(\alpha)\wedge P_{\pr_2}(\alpha))\wedge b_{A\times A}(\delta_A)\wedge b_{A\times A}(\alpha \boxtimes \alpha)\]
and then
\[b_{A\times A}(P_{\pr_1}(\alpha)\wedge P_{\pr_2}(\alpha))=b_{A\times A}(\alpha \boxtimes \alpha).\]
Therefore, we conclude that $(\wht{F},\wht{b})$ is a 1-cell of elementary doctrines since
\[ \wht{b}_{(A,\alpha)\times (A,\alpha
)}(\delta_{(A,\alpha)})= b_{A\times A}(\delta_A\wedge \alpha \boxtimes \alpha)=(R_c)_{\angbr{F(\pr_1)}{F(\pr_2)}}(\delta_{\wht{F}(A,\alpha)}).\]
Finally, it is easy to see that $(\wht{F},\wht{b})$ preserves comprehensions since every comprehension in $\mG_P $ is of the form
\[  \freccia{(A,\gamma)}{\comp{\gamma}}{(A,\alpha)}\]
where $\gamma\in P_c(A,\alpha)$, and $\comp{\gamma}$ is the identity on $A$. Then the arrow
\[ \freccia{(FA,b_A(\gamma))}{F(\comp{\gamma})}{(FA,b_A(\alpha))}\]
is $\id_{FA}$ by definition of $\wht{F}$, so it is the comprehension of $b_A(\gamma)$. 
\end{proof}
\begin{lemma}\label{prop C is well def on 2-cell}
Let $(F,b)$ and $(G,c)$ be two objects in $\ED(P,R)$ and let $\freccia{(F,b)}{\theta}{(G,c)}$ be a 2-cell in $\ED$. We define 
\[ \freccia{(\wht{F},\wht{b})}{\wht{\theta}}{(\wht{G},\wht{c})}\]
where
\[\freccia{(FA,b_A(\alpha))}{\wht{\theta}_{(A,\alpha)}}{(GA,c_A(\alpha))}\]
is $\theta_A$. Then it is a 2-cell in $\CE$.
\end{lemma}
\begin{proof}
Let $(A,\alpha)$ be an object of $\mG_P$. First, recall that we have that $b_A(\alpha)\leq R_{\theta_A}(c_A(\alpha))$ because $\theta$ is a 2-morphism. Therefore
\[\freccia{(FA,b_A(\alpha))}{\theta_A}{(GA,c_A(\alpha))}\]
is a morphism in $\mG_R$. Now, consider an element $\gamma$ of $ P_c(A,\alpha)$. Then 
\[ (R_c)_{\theta_A}(\wht{c}_A(\gamma))=R_{\theta_A}(c_A(\gamma))\wedge b_A(\alpha)
\]
by definition of the functor $R_c$. Finally, observe that $b_A(\gamma)\leq b_A(\alpha)$ since $\gamma \in P_c(A,\alpha)$, and $b_A(\gamma)\leq R_{\theta_A}(c_A(\gamma))$, and then we can conclude that
\[ \wht{b}_A(\gamma)=b_A(\gamma)\leq R_{\theta_A}(c_A(\gamma))\wedge b_A(\alpha)= (R_c)_{\theta_A}(\wht{c}_A(\gamma)).\]

\end{proof}
The previous results allow to conclude the following proposition.
\begin{proposition}
The assignment 
\[\freccia{\ED(P,R)}{\funC_{P,R}}{\CE(P_c,R_c)}\]
which maps $(F,b)$ into $(\wht{F},\wht{b})$ and a 2-cell $\freccia{(F,b)}{\theta}{(G,c)}$ into $\freccia{(\wht{F},\wht{b})}{\wht{\theta}}{(\wht{G},\wht{c})}$ is a functor and
\[\freccia{\ED}{\funC}{\CE}\]
is a 2-functor with the assignment $\funC(P)=P_c$.
\end{proposition}
\noindent
Now we prove that the 2-functor $\freccia{\ED}{\funC}{\CE}$ is 2-left adjoint to the forgetful functor $\freccia{\CE}{\funU}{\ED}$. So, we start by defining the unit and counit morphisms.

First, observe that for every elementary doctrine $P$ there is a natural embedding  
\[\duemorfismo{\mC}{P}{I_P}{{i_P}}{\mG_P}{{P_c}}\]
of elementary doctrines, where $\freccia{\mC}{I_P}{\mG_P}$ acts as $A\mapsto (A,\top_A)$, and the morphism $\freccia{P(A)}{(i_P)_A}{P_c(A,\top_A)}$ sends $\alpha\mapsto \alpha$. These 1-cells will give the unit of the 2-adjunction.

In order to define the counit, let us consider an elementary doctrine $P$ with full comprehensions. We can define a morphism

\[\duemorfismo{\mG_P}{{P_c}}{J_P}{{j_P}}{\mC}{P}\]
in $\CompAdj$ as follow: the functor $\freccia{\mG_P}{J_P}{\mC}$ sends an object $(A,\alpha)$ of $\mG_P$ to the object $A_{\alpha}$, where $A_{\alpha}$ is the domain of the comprehension $\freccia{A_{\alpha}}{\comp{\alpha}}{A}$. Given an arrow $\freccia{(A,\alpha)}{f}{(B,\beta)}$ in $\mG_P$, the arrow $\freccia{A_{\alpha}}{J_P(f)}{B_{\beta}}$ is given by the vertical arrow $(\comp{\beta}^{\ast}f) a$ of the following diagram:
\begin{equation}\label{diagram J(f)}
\xymatrix@+2pc{
A_{\alpha} \ar[rd]^{\comp{\alpha}}\ar@{-->}[d]_a\\
A_{P_f(\beta)}\ar[d]_{\comp{\beta}^{\ast}f} \ar[r]^{ \comp{P_f(\beta)}} & A\ar[d]^f\\
B_{\beta} \ar[r]_{\comp{\beta}}& B
}
\end{equation}
where $\freccia{A_{\alpha}}{a}{D}$ exists because $\alpha\leq P_f(\beta)$. Observe that by Remark \ref{remark compr has pullbacks} we have that  $a=\comp{P_{\comp{P_f(\beta)}}(\alpha)}$, and then $J_P(f)=(\comp{\beta}^{\ast}f)\comp{P_{\comp{P_f(\beta)}}(\alpha)}$. The natural transformation $j_P$ is defined by the following components: for every $(A,\alpha)$ of $\mG_P$ the arrow $\freccia{\comprehensiondoctrine{P}(A,\alpha)}{j_{(A,\alpha)}}{P(A_{\alpha})}$ acts as $\gamma\mapsto P_{\comp{\alpha}}(\gamma)$.
\begin{lemma}
With the previous assignments, $\freccia{P_c}{(J_P,j_P)}{P}$ is a 1-cell of $\CE$.
\end{lemma}
\begin{proof} 
First we prove that $\freccia{\mG_P}{J_P}{\mC}$ is a functor. Let us consider two arrows $\frecciasopra{(A,\alpha)}{f}{(B,\beta)}$ and $\frecciasopra{(B,\beta)}{g}{(C,\gamma)}$ of $\mG_P$. We need to prove that $J_P(gf)=J_P(g)J_P(f)$.
Let us consider the following diagrams
\[
\xymatrix@+2pc{
A_{\alpha} \ar[rd]^{\comp{\alpha}}\ar@{-->}[d]_a & & B_{\beta}\ar[rd]^{\comp{\beta}}\ar@{-->}[d]_b &&A_{\alpha} \ar[rd]^{\comp{\alpha}}\ar@{-->}[d]_c\\
A_{P_f(\beta)}\ar[d]_{\comp{\beta}^{\ast}f} \ar[r]^{ \comp{P_f(\beta)}} & A\ar[d]^f& B_{P_g(\gamma)}\ar[d]_{\comp{\gamma}^{\ast}g} \ar[r]^{ \comp{P_g(\gamma)}} & B\ar[d]^g &A_{P_{gf}(\beta)}\ar[d]_{\comp{\gamma}^{\ast}(gf)} \ar[r]^{ \comp{P_{gf}(\gamma)}} & A\ar[d]^{gf}\\
B_{\beta} \ar[r]_{\comp{\beta}}& B & C_{\gamma} \ar[r]_{\comp{\gamma}}& C
&C_{\gamma} \ar[r]_{\comp{\gamma}}& C}
\]
which are, $J_P(f)$, $J_P(g)$ and $J_P(gf)$ respectively. We have that $J_P(g)J_P(f)=(\comp{\gamma}^{\ast}g)b(\comp{\beta}^{\ast}f)a$ and $J_P(gf)=c(\comp{\gamma}^{\ast}(gf))$. Now notice that the following diagram 

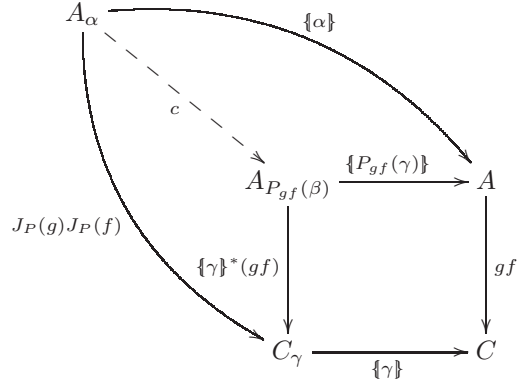
\[\xymatrix@+2pc{
A_{\alpha} \ar@/_2pc/[ddr]_{J_P(g)J_P(f)} \ar@/^2pc/[rrd]^{\comp{\alpha}}\ar@{-->}[dr]_c\\
&A_{P_{gf}(\beta)}\ar[d]_{\comp{\gamma}^{\ast}(gf)} \ar[r]^{ \comp{P_{gf}(\gamma)}} & A\ar[d]^{gf}\\
&C_{\gamma} \ar[r]_{\comp{\gamma}}& C
}\]
commutes because $$\comp{\gamma}J_P(g)J_P(f)=\comp{\gamma}(\comp{\gamma}^{\ast}g)b(\comp{\beta}^{\ast}f)a=g\comp{P_g(\gamma)}b(\comp{\beta}^{\ast}f)a=g\comp{\beta}(\comp{\beta}^{\ast}f)a$$
and then 
$$\comp{\gamma}J_P(g)J_P(f)= gf\comp{P_f(\beta)}a=gf\comp{\alpha}.$$
Then we have that
\[J_P(g)J_P(f)=J_P(gf).\] 
Moreover, it is direct to see that $J_P(\id_{(A,\alpha)}) =\id_{A_{\alpha}}$. Hence $\freccia{\mG_P}{J_P}{\mC}$ is a functor, and it is direct to show that it preserves finite products.
Now show the naturality of this assignment $j_p$.

Let $\freccia{(A,\alpha)}{f}{(B,\beta)}$ be an arrow of $\mG_P$, we have that the diagram 
\[\xymatrix@+2pc{
\comprehensiondoctrine{P}(B,\beta)\ar[d]_{j_{(B,\beta)}}\ar[r]^{\comprehensiondoctrine{P}(f)} & \comprehensiondoctrine{P}(A,\alpha)\ar[d]^{j_{(A,\alpha)}}\\
P(B_{\beta})\ar[r]_{P_{J(f)}} & P(A_{\alpha})
}\]
commutes because if $\gamma\in \comprehensiondoctrine{P}(B,\beta)$, we have that
\[P_{J_P(f)}(j_{(B,\beta)}(\gamma))=P_{(\comp{\beta}^{\ast}f)a}(P_{\comp
{\beta}}(\gamma))
\]
and by definition, this is equal to $P_{f\comp{\alpha}}(\gamma)$, which is exactly $j_{(A,\alpha)}(\comprehensiondoctrine{P}(f)(\gamma))$. Therefore $\freccia{\comprehensiondoctrine{P}}{j_P}{PJ^{\op}}$ is a natural transformation. Finally, we have to prove that the functor $J_P$ preserves comprehensions. Observe that every comprehension of an element $\alpha\in P_c(A,\beta)$ is of the form $\freccia{(A,\alpha)}{\comp{\alpha}=\id_A}{(A,\beta)}$, and then  $J_P(\comp{\alpha})= a$ where $a$ is the arrow such that the diagram
\[\xymatrix@+2pc{ A_{\alpha} \ar@{-->}[d]_a\ar[rd]^{\comp{\alpha}}\\
A_{\beta} \ar[r]_{\comp{\beta}}& A}\]
commutes. By Remark \ref{remark compr has pullbacks}, we have that $J_P(\comp{\alpha})=a=\comp{P_{\comp{\beta}}(\alpha)}$, and this is exactly the comprehension of $j_{(A,\beta)}(\alpha)=P_{\comp{\beta}}(\alpha)$. Therefore we have proved that $(J_P,j_P)$ is an 1-cell of $\CE$.
\end{proof}

\begin{theorem}\label{theorem comprehension completion}
The 2-functor $\freccia{\ED}{\funC}{\CE}$ is 2-left adjoint to the forgetful functor $\freccia{\CE}{\funU}{\ED}$. The unit of this 2-adjunction $\freccia{\id_{\PD}}{\eta}{\funU\funC}$ is given by $\eta_P=(I_P,i_P)$ and the counit $\freccia{\funC\funU}{\varepsilon}{\id_{\CompAdj}}$ is given by $\varepsilon_P=(J_P,j_P)$.
\end{theorem}
\begin{proof}
It is direct to verify that $\varepsilon$ and $\eta$ are 2-natural transformations and that $\id_{\funC}=\varepsilon\funC\circ \funC\eta$ and that $\id_{\funU}=\funU \varepsilon\circ \eta\funU$. 
\end{proof}
The 2-adjunction of Theorem  \ref{theorem comprehension completion} induces a 2-monad $\freccia{\PD}{\gromonad}{\PD}$, whose unit is given by the unit of the 2-adjunction, and whose multiplication is defined by $\mu=\varepsilon \funC$, as in 1-dimensional case.
Now we show that 2-monad $\gromonad$ is \bemph{colax-idempotent}, and that we have the equivalence of 2-categories
$$\alg{\gromonad}\equiv \CE .$$
This means that for an elementary doctrine, the structure of \emph{doctrine with full comprehensions} is a more than a structure: it is a \emph{property} in the sense of \cite{OPLS}.

We start by showing that every elementary doctrine with full comprehensions is  a $\gromonad$-algebra . 
\begin{proposition}\label{proposition P has comp imples T-alg
}
Let $\doctrine{\mC}{P}$ be an elementary doctrine of $\CompAdj$, then $(P,\varepsilon_P)$ is a $\gromonad$-algebra.
\end{proposition}
\begin{proof}
The diagram
\[\comsquare{\gromonad^2 P}{\gromonad P}{\gromonad P}{P}{\gromonad \varepsilon_P}{\mu_P}{\varepsilon_P}{\varepsilon_P}\]
commutes because $\mu_P=\varepsilon_{\funC P}$ and $\freccia{\gromonad}{\varepsilon}{\id_{\CompAdj}}$ is a 2-natural transformation. Similarly we have that the unit axiom for strict algebras is satisfied.
\end{proof}

\begin{proposition}\label{proposition P algebras imples comprehensios}
Let $(P,(F,b))$ be a $\alg{\gromonad}$. Then the doctrine $P$ has full comprehensions. Moreover $(F,b)=\varepsilon_P$.
\end{proposition}
\begin{proof}
It is direct to check that given $\alpha\in P(A)$, then $\freccia{F(A,\alpha)}{F(\id_A)}{F(A,\top_A)}$ is a full comprehension of $\alpha$, where $\freccia{(A,\alpha)}{\id_A}{(A,\top_A)}$ is the comprehension of $\alpha\in P_c(A,\top_A)$. Then the doctrine $P$ has full comprehensions and the action $(F,b)$ preserves them. Now we show that $(F,b)=\varepsilon_P$.
By the unit axiom of algebras, we have that $(F,b)\eta_P=id_P$, but since $P$ has comprehensions, we also have $\varepsilon_P \eta_P=\id_P$. Therefore we have that $(F,b)\eta_P=\varepsilon_P\eta_P$ implies that $(F,b)=\varepsilon_P$, because $F(A,\alpha)=F((A,\top_A)_{\alpha})=(F(A,\top_A))_{b_{(A,\top)}(\alpha)}=(\varepsilon_P\eta_P(A))_{\alpha}=\varepsilon_P(A,\alpha)$. Similarly one can prove that $F(f)=\varepsilon_P(f)$, because every arrow $\freccia{(A,\alpha)}{f}{(B,\beta)}$ is the unique arrow such that the following diagram commutes 
\[\xymatrix@+2PC{
(A,\alpha)\ar[r]^{\comp{\alpha}}\ar[d]_f & (A,\top_A)\ar[d]^{f}\\
(B,\beta) \ar[r]_{\comp{\beta}} & (B,\top)
}\]
since $(P_c)_{f\comp{\alpha} }(\beta)=\top_{(A,\alpha)}$, since $\alpha\leq P_f(\beta)$. Since both $F$ and $\varepsilon_P$ preserve comprehensions, and since $F(\comp{\alpha})=\varepsilon_P(\comp{\alpha})$ and $F(\frecciasopra{(A,\top_A)}{f}{(B,\top_B)})=\varepsilon_P(\frecciasopra{(A,\top_A)}{f}{(B,\top_B)})$, then $F(\frecciasopra{(A,\alpha)}{f}{(B,\beta)})$ must be equal to the arrow $\varepsilon_P(\frecciasopra{(A,\alpha)}{f}{(B,\beta)})$ (by the unicity of the mediating arrow in the universal property of comprehensions). Hence $F=\varepsilon_P$. Finally it is direct to check that $b=j_P$.
\end{proof}

\begin{proposition}\label{proposition every morf extend uniquile to T-morf}
Let $P$ and $R$ be two doctrines of $\CompAdj$, and let $\freccia{P}{(F,b)}{R}$ be a 1-cell of $\PD$, then there exists a unique 2-cell $\tau$ such that $\freccia{(P,\varepsilon_P)}{((F,b),\tau)}{(R,\varepsilon_R)}$ is a colax morphism of $\alg{\gromonad}$.
\end{proposition}
\begin{proof}
Consider the square
\[ \comsquare{P_c}{R_c}{P}{R.}{\gromonad (F,b)}{\varepsilon_P}{\varepsilon_R}{(F,b)}\]
Let $(A,\alpha)$ be an object of $\mG_P$. Then we have that \[\varepsilon_R \gromonad(F,b)(A,\alpha)=(FA)_{b_A(\alpha)}\]
and
\[(F,b)\varepsilon_P (A,\alpha)=F(A_{\alpha}).\]
We define $\tau_{(A,\alpha)}$ as the morphism
\[ \xymatrix@+2pc{
(FA)_{b_A(\alpha)} \ar[r]^{\comp{b(\alpha)}} & FA \\
& F(A_{\alpha}) \ar[u]_{F(\comp{\alpha})} \ar[lu]^{\tau_{(A,\alpha)}}
}\]
which exists by the universal property of comprehensions, because
\[ R_{F(\comp{\alpha})}(b_A(\alpha))=b_{A_{\comp{\alpha}}}P_{\comp{\alpha}}(\alpha)=\top.\]
Now we show that the $\tau$  is a natural transformation $\duefreccia{FJ_P}{\tau}{J_R\wht{F}}$. Let us consider an arrow $\freccia{(A,\alpha)}{f}{(B,\beta)}$ of the category $\mG_P$. Then the diagram 
\[\xymatrix@+2pc{
F(A_{\alpha})\ar[dd]_{FJ_P(f)}\ar[rr]^{F(\comp{\alpha})} \ar[rd]_{\tau_{(A,\alpha)}}& &FA\ar[dd]^{F(f)}\\
& (FA)_{b_A(\alpha)}\ar[dd]_(0.3){J_R\wht{F}(f)}\ar[ru]_{\comp{b_A(\alpha)}}\\
F(B_{\beta})\ar[dr]_{\tau_{(B,\beta)}}\ar[rr]|!{[d];[r]}\hole_(0.55){\;\;\;\;\;\;\;F(\comp{\beta})} & &FB\\
& (FB)_{b_B(\beta)}\ar[ru]_{\comp{b_B{\beta}}}.
}\]
commutes, because every triangle commutes and the right and back squares commute, hence, using the fact that comprehensions are mono, we can show that the left square commutes. Therefore, $\tau$ is a natural transformation. Moreover we have that 
\[
\xymatrix@+1pc{
&P(A_{\alpha})\ar[r]^{b_{A_{\alpha}}}& R(FA_{\alpha})\ar[dd]^{R(\tau_{(A,\alpha)})}\\
P_c(A,\alpha) \ar[ur]^{j_P} \ar[dr]_{\wht{b}_{(A,\alpha)}}\\ &R_c(FA,b_A(\alpha))\ar[r]_{j_R}& R((FA)_{b_A(\alpha)})
}\]
commutes, and hence we can conclude that $\tau$ is a 2-cell of $\CE$.

Finally, it is direct to show that $((F,b), \tau)$ satisfies the coherence axioms of colax morphisms of algebras. For example, we have that the following axiom is satisfied 
\[\xymatrix@+1pc{
P \ar[d]_{\eta_P}\ar[r]^{(F,b)} &R\ar[d]^{\eta_R}&&P\ar[dd]_{1_P}\ar[r]^{(F,b)}& R\ar[dd]^{1_B}\\
P_c \ar[d]_{\varepsilon_P} \ar[r]^{(\wht{F},\wht{b})} \xtwocell[r]{}<>{^<5>\tau}&R_c \ar[d]^{\varepsilon_R}\ar@{}[rr]|{=}&& \\
P \ar[r]_{(F,b)} & R&& P\ar[r]_{(F,b)} &R
}\]
because, when $\alpha=\top_A$, then we have that $\tau_{(A,\top_A)}=\id_{FA}$.
Now we show that this $\tau$ is unique. Let us consider another 2-cell $\duefreccia{(F,b)\varepsilon_P}{\theta}{\varepsilon_R\mT_c(F,b)}$ such that $((F,b),\theta)$ is a colax-morphism 
\[\xymatrix@+1pc{
P_c \ar[r]^{(\wht{F},\wht{b})} \ar[d]_{\varepsilon_P} \xtwocell[r]{}<>{^<5>\;\theta
}& R_c \ar[d]^{\varepsilon_R}\\
P \ar[r]_{(F,b)} &R
}\]
of $\mT_c$ algebras. Then it must satisfy the following condition
\[\xymatrix@+1pc{
P \ar[d]_{\eta_A}\ar[r]^{(F,b)} &R\ar[d]^{\eta_B}&&P\ar[dd]_{1_P}\ar[r]^{(F,b)}& R\ar[dd]^{1_B}\\
P_c \ar[d]_{\varepsilon_P} \ar[r]^{(\ltil{F},b)} \xtwocell[r]{}<>{^<5>\theta}&R_c \ar[d]^{\varepsilon_R}\ar@{}[rr]|{=}&& \\
P \ar[r]_{(F,b)} & R&& P\ar[r]_{(F,b)} &R
}\]
and this means that $\theta_{(A,\top_A)}=\id_{FA}$. Therefore, since $\theta$ is a natural transformation from $FJ_P$ to $J_R\wht{F}$, then the following diagram
\[\xymatrix@+2pc{
F(A_{\alpha})\ar[d]_{\theta_{(A,\alpha)}} \ar[r]^{F(\comp{\alpha})}&FA\ar[d]^{\theta_{(A,\top)}}\\
(FA)_{b_A(\alpha)} \ar[r]_{\comp{b_A(\alpha)}}& FA
}\]
commutes and since $\theta_{(A,\top_A)}=\id_{FA}$, then we have that $\theta_{(A,\alpha)}$ must be $\tau_{(A,\alpha)}$ because, by definition, $\tau_{(A,\alpha)}$ is the unique arrow such that the diagram
\[ \xymatrix@+2pc{
(FA)_{b_A(\alpha)} \ar[r]^{\comp{b(\alpha)}} & FA \\
& F(A_{\alpha}) \ar[u]_{F(\comp{\alpha})} \ar[lu]^{\tau_{(A,\alpha)}}
}\]
commutes. Hence $\theta=\tau$.
\end{proof}
\begin{corollary}
Let $P$ and $R$ be two doctrines of $\CompAdj$, and let $\freccia{P}{(F,f)}{R}$ be an invertible 1-cell of $\CompAdj$, then $\tau$ is the identity.
\end{corollary}
\begin{proof}
By Proposition \ref{proposition every morf extend uniquile to T-morf} $\tau$ exists and it is unique, and since $\varepsilon$ is 2-natural, it must be the identity.
\end{proof}
\begin{remark}
Observe that if we have a $\gromonad$-morphism $((F,b),\gamma)$
\[\xymatrix@+2pc{
\gromonad P \ar[d]_{\varepsilon_P} \ar[r]^{\gromonad (F,f)} \xtwocell[r]{}<>{^<6>\gamma}& \gromonad R \ar[d]^{\varepsilon_R}\\
P \ar[r]_{(F,f)} &R
}\] 
then, since $\gamma$ is invertible, we have that $(F,b)$ is a 1-cell of $\CompAdj$, so by the previous corollary, $\gamma$ must be the identity.
\end{remark}
Combining previous results, we directly show that the comprehensions completions is 2-monadic.
\begin{theorem}\label{theorem comp. comp. is monadic}
The 2-category $\alg{\gromonad}$ of strict algebras, algebras morphisms and $\gromonad$-transformation is 2-equivalent to the 2-category $\CompAdj$.
\end{theorem}
Finally, again by directly applying the previous results, we can prove the following theorem.
\begin{theorem}
The 2-monad $\freccia{\PD}{\gromonad}{\PD}$ is colax-idempotent.
\end{theorem}
\begin{proof}
It follows by Proposition \ref{proposition every morf extend uniquile to T-morf} and by Proposition \ref{proposition P algebras imples comprehensios}.
\end{proof}
Notice that, in particular, the previous theorem implies that the 2-monad $\mT_c$ is property-like, and so we can conclude that having full comprehensions is not only a structure, but it is a \emph{property} of an elementary doctrine.
\begin{remark}
Observe that having a \emph{choice of comprehensions} in the doctrine plays a fundamental role in the development of the results of this section. For example, in the definition of the counit of the 2-adjunction. However, all the results we presented can be generalized if we do not assume any choice of comprehensions. In this case, the functor $\mC$ remains a 2-functor, the unit is a 2-natural transformation, but the counit becomes a pseudo-natural transformation, and then the monad $\mT_c$ is just a pseudo-monad, and not a 2-monad. Then, all the results can be reformulated in terms of pseudo-monads.

\end{remark}
\section{Elementary doctrines with comprehensive diagonals}\label{section comprensive diag}
The notion of \emph{comprehensive diagonals}, together with the construction called \emph{extensional collapse} of an elementary doctrine, was employed by Maietti and Rosolini in \cite{QCFF} to obtain "extensional” models of constructive theories.

We start this section by recalling this notion, which is a special case of comprehension, see Definition \ref{def comrehension}, and then presenting the free construction that forces the equalities of an elementary doctrine to be "extensional".
\begin{definition}\label{def comprehensive diagonals}
An elementary doctrine $\doctrine{\mC}{P}$ \bemph{has comprehensive diagonals} if every diagonal arrow
$  \freccia{A}{\angbr{\id_A}{\id_A}}{A\times A}$ is the comprehension of $\delta_A$.
\end{definition}
The denote by $\CED$ the 2-category whose objects are elementary doctrines with comprehensive diagonals, and whose 1-cells and 2-cells are the same of $\ED$. 

Now we recall the free construction which freely adds comprehensive diagonals to an elementary doctrine

\noindent
Let $\doctrine{\mC}{P}$ be an elementary doctrine, we define $\mX_P$ the \bemph{extensional collapse} of $P$:
\begin{itemize}
\item \textbf{the objects} of $\mX_P$ are the objects of $\mC$;
\item \textbf{a morphism} $\freccia{A}{[f]}{B}$ is an equivalence class of morphisms $\freccia{A}{f}{B}$ such that $\delta_A \leq_{A\times A} P_{f\times f}(\delta_B)$ with respect to the equivalence $f\sim f'$ when $\delta_A\leq_{A\times A} P_{f\times f'}(\delta_B)$.
\end{itemize}
The indexed inf-semilattice $\doctrine{\mX_P}{P_x}$ will
be given by $P$ itself: indeed for every $A$ in $\mC$, $P_x(A)=P(A)$
and for every $\freccia{A}{[f]}{B}$, $P_x([f])=P(f)$ as one shows that
$P(f)=P(f')$ when $f\sim f'$. See \cite[Lemma 5.5]{EQC}.

With the previous assignments the functor $\doctrine{\mX_P}{P_x}$  is an elementary doctrine with comprehensive diagonals. Now we show that, as for the case of ordinary comprehensions, the assignment $P\mapsto P_x$ can be extended to 2-functor 
\[\freccia{\ED}{\funD}{\CED}\]
and we start defining how it acts on the 1-cells and 2-cells in $\ED$.

Let $\doctrine{\mC}{P}$ and $\doctrine{\mD}{R}$ be elementary doctrines, and consider a 1-cell $(F,b)$:
\[\duemorfismo{\mC}{P}{F}{b}{\mD}{R}\]
Let $(\ltil{F},b)$ be the pair where
\begin{itemize}
\item $\ltil{F}(A)$ is $F(A)$ for every $A\in \mX_P$;
\item $\ltil{F}([f])$ is $[F(f)]$ for every $\freccia{A}{[f]}{B}$.
and $b$ remains the same.
\end{itemize}

\begin{lemma}
$(\ltil{F},b)$ is a 1-morphism in $\CED$.
\end{lemma}
\begin{proof}
First we prove that $\freccia{\mX_P}{\tilde{F}}{\mX_R}$ is a
well-defined functor.
If $\freccia{A}{f}{B}$ and $\freccia{A}{g}{B}$ are a morphism in
$\mC$, such that $ \delta_A \leq P_{g \times f }(\delta_B)$, then we
have 
\[b_{A \times A}(\delta_A)\leq b_{A \times A}(P_{g \times f }(\delta_B))\]
Since $b$ is a natural transformation, the following diagram commutes
\[ \comsquare{P(B\times B)}{P(A\times A)}{RF(B\times B)}{RF(A\times A).}{P_{g\times f}}{b_{B\times B}}{b_{A\times A}}{R_{F(g\times f)}}\]
Hence we have
\[b_{A \times A}(\delta_A)\leq R_{F(g \times f )}(b_{B \times B}(\delta_B)).\]
By definition,
$b_{A \times
  A}(\delta_A)=R_{\angbr{F(\pr_1)}{F(\pr_2)}}(\delta_{F(B)})$, thus
\[R_{\angbr{F(\pr_1)}{F(\pr_2)}}(\delta_{F(A)})\leq R_{\angbr{F(\pr'_1)}{F(\pr'_2)}\circ F(g\times f)}(\delta_{F(B)})\]
where $\freccia{A\times A}{\pr_i}{A}$ and $\freccia{B\times B}{\pr'_i}{B}$ are the projections.
Finally 
\[F(g \times f) \circ \angbr{F(\pr_1)}{F(\pr_2)}^{-1}=\angbr{F(\pr'_1)}{F(\pr'_2)}\circ F(g)\times F(f),\]
so
\[ \delta_A\leq R_{F(g)\times F(f)}(\delta_B).\]
It is now easy to check that $\tilde{F}$ is a functor from
$\mX_P$ to $\mX_R$.
Then, we have that $(\tilde{F}, b)$ is a 1-cell observing that
\[ b_{A\times A}(\delta_A)=(R_x)_{\angbr{\ltil{F}([\pr_1])}{\ltil{F}([\pr_2])}}(\delta_{\ltil{F}(B)})\]
because $\ltil{F}([\pr_i])=[F(\pr_i)]$, $\ltil{F}(B)=F(B)$ by definition of $\ltil{F}$, and 
\[ \angbr{\ltil{F}([\pr_1])}{\ltil{F}([\pr_2])}=[\angbr{F(\pr_1)}{F(\pr_2)}]\]
by \cite[Lemma 5.4]{EQC}, and 
\[ (R_x)_{\angbr{\ltil{F}([\pr_1])}{\ltil{F}([pr2])}}=(R_x)_{[\angbr{F(\pr_1)}{F(\pr_2)}]}=R_{\angbr{F(\pr_1)}{F(\pr_2)}}. \]
\end{proof}
\begin{lemma}
Given a 2-cell $\duefreccia{(F,b)}{\theta}{(G,c)}$, where $(F,b)$ and
$(G,c)$ are 1-cells in $\ED(P,R)$, we define
$\freccia{\ltil{F}}{\ltil{\theta}}{\ltil{G}}$ as the natural
transformation with $\ltil{\theta}_A=[\theta_A]$. Then $\ltil{\theta}$ is a 2-cell of $\CED$.
\end{lemma}
\begin{proof} Since $\theta$ it is a 2-cell
in $\ED$ we have that
\[ b_A(\alpha)\leq_{F(A)} R_{\theta_A}(c_A(\alpha)).\]
Hence, by definition of $R_x$ and $\ltil{F}$, we have
\[ R_{\theta_A}(c_A(\alpha))=(R_x)_{[\theta_A]}(c_A(\alpha))=(R_x)_{\ltil{\theta}_A}(c_A(\alpha)),\]
so 
\[ b_A(\alpha)\leq_{\ltil{F}(A)}(R_x)_{\ltil{\theta}_A}(c_A(\alpha)).\]
Then $\ltil{\theta}$ is a 2-cell of $\ED$, and hence of $\CED$.  
\end{proof}
\begin{proposition}\label{functor D_P,R}
Let $\doctrine{\mC}{P}$ and $\doctrine{\mD}{R}$ be elementary doctrines. The map
\[\freccia{\ED(P,R)}{\funD_{P,R}}{\CED(P_x,R_x)}\]
such that $\funD_{P,R}(F,b)=(\ltil{F},b)$ and $\funD_{P,R}(\theta)=\ltil{\theta}$ is a functor and
\[\freccia{\ED}{\funD}{\CED} \]
is a 2-functor with the assignment $\funD(P)=P_x$.
\end{proposition}
We prove that the 2-functor $\freccia{\ED}{\funD}{\CED}$ is left adjoint to the forgetful 2-functor $\freccia{\CED}{\funU}{\ED}$.

First, observe that for every elementary doctrine $P$ there is a natural embedding  
\[\duemorfismo{\mC}{P}{K_P}{{k_P}}{\mX_P}{{P_x}}\]
of elementary doctrines, where $\freccia{\mC}{K_P}{\mX_P}$ is the quotient functor, and $k$ is the identity.

Similarly, if $P$ is an elementary doctrine with comprehensive diagonals, we can define a 1-cell
\[\duemorfismo{\mX_P}{{P_x}}{T_P}{{t_P}}{\mC}{P}\]
where $\freccia{\mX_P}{T_P}{\mC}$ is the identity on the objects, and it sends $[f]\mapsto f$, and $t_P$ is the identity. Notice that $T_P$ is a well defined functor, because if $\doctrine{\mC}{P}$ has comprehensive diagonals, then $f\sim g$ implies $f=g$. In details, let $\freccia{A}{f}{B}$ and $\freccia{A}{g}{B}$ be morphisms such that $\delta_A\leq P_{f\times g}(\delta_B)$. Then we have that $$\top_A\leq P_{\Delta_A} (P_{f\times g }(\delta_B))=P_{\angbr{f}{g}}(\delta_B).$$
Thus, there exists a unique morphism $\freccia{A}{h}{B}$ such that the following diagram
\[\xymatrix{
B \ar[rr]^{\Delta_B} && B\times B\\
& A \ar@{-->}[ul]^h \ar[ur]_{\langle f,g \rangle}
}\]
commutes. Hence, if $P\in \CED$ then we have that $f\sim g$ if and only if $f=g$.
Thus, we can define the 1-cell $\freccia{P_x}{(T_P,t_P)}{P}$ as before.

\begin{remark}\label{rem ps idem}
Notice that  if $P$ has comprehensive diagonal, then $\varepsilon_P$ and $\eta_P$ are isomorphism.
\end{remark}
\begin{theorem}\label{theorem diag comprehension completion}
The 2-functor $\freccia{\ED}{\funD}{\CED}$ is 2-left adjoint to the forgetful functor $\freccia{\CED}{\funU}{\ED}$. The unit of this 2-adjunction $\freccia{\id_{\PD}}{\eta}{\funU\funD}$ is given by $\eta_P=(K_P,k_P)$ and the counit $\freccia{\funD\funU}{\varepsilon}{\id_{\CED}}$ is given by $\varepsilon_P=(T_P,t_P)$.
\end{theorem}
The 2-adjunction of Theorem  \ref{theorem diag comprehension completion} induces a 2-monad $\freccia{\PD}{\mT_d}{\PD}$, whose unit is given by the unit of the 2-adjunction, and whose multiplication is defined by $\mu=\varepsilon \funD$, as in 1-dimensional case.

As for the case of the comprehension completion, we show that 2-monad $\mT_d$ is \bemph{pseudo-idempotent}, and that we have the equivalence of 2-categories
$$\alg{\mT_d}\equiv \CED .$$
However, in this case, it is immediate to show that the 2, because by Remark \ref{rem ps idem} we have that the multiplication of the 2-monad $\mT$ is invertible.
\begin{theorem}\label{theorem T is lax-idempotent}
The 2-monad $\freccia{\ED}{\mT_d}{\ED}$ is pseudo-idempotent.
\end{theorem}
\begin{proof}
It follow from Remark \ref{rem ps idem}.
\end{proof}
The previous result means that also having comprehensive diagonals is a property of an elementary doctrine.
\begin{proposition}\label{proposition P has comp diag imples T-alg
}
Let $\doctrine{\mC}{P}$ be an elementary doctrine of $\CED$, then $(P,\varepsilon_P)$ is a $\mT_d$-algebra.
\end{proposition}
\begin{proof}
The diagram
\[\comsquare{\mT_d^2 P}{\mT_d P}{\mT_d P}{P}{\mT_d \varepsilon_P}{\mu_P}{\varepsilon_P}{\varepsilon_P}\]
commutes because $\mu_P=\varepsilon_{\funD P}$ and $\freccia{\mT_d}{\varepsilon}{\id_{\CED}}$ is a 2-natural transformation. Similarly we have that the unit axiom for strict algebras is satisfied.
\end{proof}

\begin{proposition}\label{remark (P,a) pseudo-T-algebra implies P in CED}
Let $(P,(F,b))$ be a $\alg{\mT_d}$. Then the doctrine $P$ has comprehensive diagonals and $(F,b)=\varepsilon_P$.
\end{proposition}
\begin{proof}
Since $(P,(F,b))$ is a $\mT_d$-algebra, the following diagram, i.e. the identity axiom holds
\[\xymatrix@+1pc{
P\ar[rd]_{1_P}\ar[r]^{\eta_P} & \mT_dP \ar[d]^{(F,b)}\\
& P
}\]
commutes. Since $P_x$ has comprehensive diagonals and it acts on morphisms as $P$, i.e. $P_x([f])=P_f$ it is direct to show that $P$ has comprehensive diagonals, which are given by morphisms of the form $F([\Delta_A])$. Moreover, by unit axiom, $F$ must be the identity on the objects of $\mX_P$, and $F([f])=F\eta_P (f)=f$, i.e. $F=\varepsilon_P.$ So, let $\freccia{C}{f}{A\times A}$ be a morphism of $\mC$ such that $\top_C\leq P_f(\delta_A)$.
\end{proof}

\begin{theorem}
We have the following equivalence of categories
\[ \alg{\mT_d} \cong \CED\]
\end{theorem}
Therefore, again as in the case of the comprehension completion, we have that the comprehensive diagonal completion is 2-monadic.

\section{Elementary doctrines with quotients}\label{section doctrine with quot}
In this section we consider the completion with quotients of an
elementary doctrine introduced in \cite{QCFF,EQC}, and we show that this construction is 2-monadic and that its 2-monad is lax-idempotent.

Given an elementary doctrine $\doctrine{\mC}{P}$, an object $A$ in
$\mC$, and an object $\rho$ in $P(A\times A)$, we say that $\rho$ is a
$P$-\bemph{equivalence relation on} $A$ if it satisfies
\begin{itemize}
\item \bemph{reflexivity:} $\delta_A\leq \rho$;
\item \bemph{symmetry:} $\rho\leq P_{\angbr{\pr_2}{\pr_1}}(\rho)$, for $\freccia{A\times A}{\pr_1,\pr_2}{A}$ the first and second projection, respectively;
\item \bemph{transitivity:} $P_{\angbr{\pr_1}{\pr_2}}(\rho)\wedge P_{\angbr{\pr_2}{\pr_3}}(\rho)\leq P_{\angbr{\pr_1}{\pr_3}}(\rho)$ for \[\freccia{A\times A\times A }{\pr_1,\pr_2,\pr_3}{A}\] the first, second, and third projection, respectively.
\end{itemize}

For an elementary doctrine $\doctrine{\mC}{P}$, the object $\delta_A$
is a $P$-equivalence relation, and for every morphism
$\freccia{A}{f}{B}$, the functor
\[\freccia{P(B\times B)}{P_{f\times f}}{P(A\times A)}\] 
takes a $P$-equivalence relation $\sigma$ on
$B$ to a $P$-equivalence relation on $A$. The $P$-\bemph{kernel
equivalence relation of} $\freccia{A}{f}{B}$ is the object 
$P_{f\times f}(\delta_B)$, which  is a $P$-equivalence relation on $A$. 

\begin{definition}
Let $\doctrine{\mC}{P}$ be an elementary doctrine, and let $\rho$ be
an $P$-equivalence relation on $A$. A $P$-\bemph{quotient} of $\rho$
is a morphism $\freccia{A}{q}{C}$ in $\mC$ such that 
$P_{q\times q}(\delta_C)\geq \rho$ and for every morphism
$\freccia{A}{f}{Z}$ 
such that $P_{f\times f }(\delta_Z)\geq \rho$, there exists a unique
morphism $\freccia{C}{g}{Z}$ such that $g\circ q=f$.
\end{definition}
A quotient  $\freccia{A}{q}{C}$ of $\rho$ is said \bemph{effective} if $P_{q\times q}(\delta_D)=\rho$. 
We say that such a $P$-quotient is \bemph{stable} if in every pullback
\[\comsquare{A'}{C'}{A}{C}{q'}{f'}{f}{q}\] 
in $\mC$, the morphism $\freccia{A'}{q'}{C'}$ is a P-quotient.

\begin{definition}
Given an elementary doctrine $\doctrine{\mC}{P}$ and a $P$-equivalence relation $\rho$ on an object $A$ in $\mC$, the partial order of \bemph{descent data} $\des_{\rho}$ is the suborder of $P(A)$ of those $\alpha$ such that
\[ P_{\pr_1}(\alpha)\wedge \rho \leq P_{\pr_2}(\alpha)\]
where $\freccia{A\times A}{\pr_1,\pr_2}{A}$ are projections.
\end{definition}

\begin{remark}
Let $\doctrine{\mC}{P}$ be an elementary doctrine and consider the $P$-kernel $\rho=P_{f\times f}(\delta_B)$, for $\freccia{A}{f}{B}$. The functor $\freccia{P(B)}{P_f}{P(A)}$ takes values in $\des_{\rho}\subseteq P(A)$.
\end{remark}
\begin{definition}
Let $\doctrine{\mC}{P}$ be an elementary doctrine. An arrow $\freccia{A}{f}{B}$ is called \bemph{of effective descent} if the functor $\freccia{P(B)}{P_f}{\des_{\rho}}$, where $\rho=P_{f\times f}(\delta_B)$ is an isomorphism.
\end{definition}

Consider the 2-full 2-subcategory $\QED$ of $\ED$ whose objects are
the elementary doctrines 
$\doctrine{\mC}{P}$ with stable effective quotients of $P$-equivalence relations and of effective descent. 1-cells of $\QED$ are those 1-cells of $\ED$ which preserve quotients, and the 2-cells of $\QED$ are the same of $\ED$. 

As in the case of comprehensions, from now on we consider doctrines with a choice of quotients.
 
Let $\doctrine{\mC}{P}$ be an elementary doctrine, and consider the
category $\mR_P$ of \bemph{P-equivalence relation}:
\begin{itemize}
\item \textbf{an object} of $\mR_P$ is a pair $(A,\rho)$ such that $\rho$ is a $P$-equivalence relation on $A$;
\item \textbf{a morphism} $\freccia{(A,\rho)}{f}{(B,\sigma)}$ is a morphism $\freccia{A}{f}{B}$ such that $\rho\leq P_{f\times f}(\sigma)$.
\end{itemize}
The indexed poset $\doctrine{\mR_P}{P_q}$ will be given by the categories of descent data:
\[ P_q(A,\rho)=\des_{\rho} \]
and for every morphism $\freccia{(A,\rho)}{f}{(B,\sigma)}$ we define
\[ P_q(f)=P(f)\]
This is a well defined elementary doctrine, see \cite[Lemma 4.2]{EQC},
and it has descent quotients of $P$-equivalence relations, see
\cite[Lemma 4.4]{EQC}.

Following the structure of Sections \ref{section comprensive diag} and \ref{sec doctrine with comp} we prove that the assignment $\funQ(P)=P_q$ can be extended to 2-functor 
\[\freccia{\ED}{\funQ}{\QED}\]
and we start defining how it acts on the 1-cells and 2-cells in $\ED$.

Let $\doctrine{\mC}{P}$ and $\doctrine{\mD}{R}$ be elementary doctrines, and consider a 1-cell $(F,b)$:
\[\duemorfismo{\mC}{P}{F}{b}{\mD}{R}\]
We want to prove that the pair $(\ovln{F},\ovln{b})$ where:
\begin{itemize}
\item $\ovln{F}(A,\rho)$ is $(FA,\RFinverso{b_{A\times A}(\rho))}$ for every $A\in \mR_P$;
\item $\ovln{F}(f)$ is $F(f)$ for every $\freccia{(A,\rho)}{f}{(B,\sigma)}$;
\item $\ovln{b}$ is $b$ restricted to the categories of descent data;
\end{itemize}
is a 2-morphism in $\QED$:
\[\xymatrix{
\mR_P^{op} \ar[rrd]^{P_q}_{}="a" \ar[dd]_{\ovln{F}^{op}}\\
&& \infsl\\
\mR_R^{op}  \ar[rru]_{R_q}^{}="b"
\ar_{\ovln{b}}  "a";"b"}\]

\begin{lemma}\label{bar_F is well def on objects}
Let $(A,\rho)$ be an object in $ \mR_P$ and let $\freccia{A\times A}{\pr_1,\pr_2}{A}$ be the two projections. Then $\RFinverso{b_{A\times A}(\rho)}$ is a $P$-equivalence relation on $FA$.
\end{lemma}
\begin{proof}
Reflexivity: $\rho$ is an equivalence relation on $A$ implies $b_{A\times A}(\delta_A)\leq b_{A\times A}(\rho)$ and by definition of $b_{A\times A}$ we have $\RFnormale{\delta_{FA}}\leq b_{A\times A}(\rho)$ .
Since $F$ preserves products $\angbr{F(\pr_1)}{F(\pr_2)}$ is an isomorphism. So
\[ \delta_{FA}\leq\RFinverso{b_{A\times A}(\rho)}.\]
Symmetry and transitivity are proved similarly.
\end{proof}
\begin{lemma}\label{bar_F is well def on morphisms}
Let $\freccia{(A,\rho)}{f}{(B,\sigma)}$ be a morphism in $\mR_P$, and let $\freccia{A\times A}{\pr_i}{A}$ and $\freccia{B\times B}{\pr_i'}{B}$, $i=1,2$ be the projections. Then \[\freccia{(FA, \RFinverso{b_{A\times A}(\rho)})}{F(f)}{(FB,R_{\langle F(\pr_1'),F(\pr_2') \rangle^{-1}}(b_{B\times B}(\sigma)))}\] is a morphism in $\mR_R$.
\end{lemma}
\begin{proof}
Since $\freccia{(A,\rho)}{f}{(B,\sigma)}$ is a 1-cell, $\rho \leq P_{f\times f}(\sigma)$. Thus \[b_{A\times A}(\rho)\leq b_{A\times A}(P_{f\times f}(\sigma))=R_{F(f\times f)}(b_{B\times B}(\sigma)).\]
Hence
\[ \RFinverso{b_{A\times A}(\rho)}\leq \RFinverso{R_{F(f\times f)}(b_{B\times B}(\sigma))}.\]
Since
\[ F(f\times f)\circ \angbr{F(\pr_1)}{F(\pr_2)}^{-1}=\angbr{F(\pr_1')}{F(\pr_2')}^{-1}\circ F(f)\times F(f) \]
it is
\[ \RFinverso{b_{A\times A}(\rho)}\leq R_{F(f)\times F(f)}(R_{\langle F(\pr_1'),F(\pr_2') \rangle^{-1}}(b_{B\times B}(\sigma))).
\]

\end{proof}

\begin{remark}\label{bar_b=b}
Consider $(A,\rho)\in \mR_P$, if $\alpha\in {\des}_{\rho}$ then \[b_A(\alpha)\in {\des}_{\RFinverso{b_{A\times A}(\rho)}}.\]
\end{remark}

\begin{corollary}
Given $(F,b)\in \ED(P,R)$ then $(\ovln{F},\ovln{b})\in \QED(P_q,R_q)$ .
\end{corollary}
\begin{proof}
By Remark \ref{bar_b=b} and \cite[Lemma 4.2]{EQC}
\[ b_{A\times A}(\rho)=\RFnormale{\RFinverso{b_{A\times A}(\rho)}}\]
So
\[\ovln{b}_{(A,\rho)\times (A,\rho)}(\delta_{(A,\rho)})=(R_q)_{\angbr{\ovln{F}(\pr_1)}{\ovln{F}(\pr_2)}}(\delta_{\ovln{F}(A,\rho)}).\]
By Lemma \ref{bar_F is well def on morphisms} and Lemma \ref{bar_F is
  well def on objects} we can conclude that $(\ovln{F},\ovln{b})\in
\ED(P_q,R_q)$. It remains to verify that $\ovln{F}$ preserves all the
quotients.

Consider a $P_q$-equivalence relation $\tau$ on $(A,\rho)$. A $P_q$-quotient of $\tau$ is
\[ \freccia{(A,\rho)}{\id_A}{(A,\tau)}\]
and
\[ \freccia{(FA,\RFinverso{b_{A\times A}(\rho)}}{\id_{FA}}{(FA,\RFinverso{b_{A\times A}(\tau)}}\]
is a $R_q$-quotient of $\RFinverso{b_{A\times A}(\tau)}$. So $\ovln{F}$ preserves quotients, and $(\ovln{F},b)$ is a 1-cell in $\QED$.
\end{proof}

\begin{proposition}\label{prop Q is well define on 2-cells}
Let $\theta$ be a morphism in $\ED(P,R)$
\[ \freccia{(F,b)}{\theta}{(G,c)}.\]
Then $\theta$ is also a morphism in $\QED(P_q,R_q)$
\[\freccia{(\ovln{F},\ovln{b})}{\theta}{(\ovln{G},\ovln{c})}.\]
\end{proposition}
\begin{proof}
We must prove that for every $(A,\rho)\in \mR_P$ 
\[\freccia{(FA,\RFinverso{b_{A\times A}(\rho)})}{\theta_A}{(GA,\RGinverso{c_{A\times A}(\rho)})}\]
is a morphism in $\mR_R$. Indeed, by definition of 2-morphism we have $b_{A\times A}(\rho)\leq R_{\theta_{A\times A}}(c_{A\times A}(\rho))$
then
\[ \RFinverso{b_{A\times A}(\rho)}\leq \RFinverso{R_{\theta_{A\times A}}(c_{A\times A}(\rho))}\]
and, since $ \theta$ is a natural transformation, 
\[ \RFinverso{b_{A\times A}(\rho)}\leq R_{\theta_{A}\times \theta_A}(\RGinverso{(c_{A\times A}(\rho))}).\]
Finally for every $\alpha\in \des_{\RFinverso{b_{A\times A}(\rho)}}$ we have
\[\ovln{b}_A(\alpha)\leq (R_q)_{\theta_A}(\ovln{c}_A(\alpha))\]
because $\ovln{b}_A(\alpha)=b_A(\alpha)$, $\ovln{c}_A(\alpha)=c_A(\alpha)$ and $R_q(\theta_A)=R(\theta_A)$.
\end{proof}

\begin{proposition}
The assignment 
\[\freccia{\ED(P,R)}{\funQ_{P,R}}{\QED(P_q,R_q)}\]
which maps $(F,b)$ into $(\ovln{F},\ovln{b})$ and a 2-cell $\freccia{(F,b)}{\theta}{(G,c)}$ into $\freccia{(\ovln{F},\ovln{b})}{\theta}{(\ovln{G},\ovln{c})}$ is a functor and
\[ \freccia{\ED}{\funQ}{\QED}\]
is a 2-functor with the assignment $\funQ(P)=P_q$.
\end{proposition}

Now we prove that the 2-functor $\freccia{\ED}{\funQ}{\QED}$ is left adjoint to the forgetful 2-functor. 

To simplify the notation, given an object $A$ and a $P$-equivalence relation $\rho$, the quotient of $\rho$ is denoted by $\freccia{A}{q_{\rho}}{\qot{A}{\rho}}$. Observe that, as in the case of comprehensions, we assume that an elementary doctrine with quotients is equipped with a choice of quotients.

First, observe that for every elementary doctrine $P$ there is a natural embedding  
\[\duemorfismo{\mC}{P}{L_P}{{l_P}}{\mR_P}{{P_q}}\]
of elementary doctrines, where $\freccia{\mC}{L_P}{\mR_P}$ acts as $A\mapsto (A,\delta_A)$, and $\freccia{P(A)}{(l_P)_A}{P_q(A,\delta_A)}$ sends $\alpha\mapsto \alpha$. It is direct to check that  $(L_P,l_P)$ is a 1-cell of elementary doctrines.

Given an elementary doctrine $P$ of $\QED$, we can define a morphism
\[\duemorfismo{\mR_P}{{P_q}}{V_P}{{v_P}}{\mC}{P}\]
in $\QED$ as follow: the functor $\freccia{\mR_P}{V_P}{\mC}$ sends an object $(A,\rho)$ of $\mR_P$ to the object $\qot{A}{\rho}$ of, and  an arrow $\freccia{(A,\rho)}{f}{(B,\sigma)}$ in $\mR_P$ is sent to the arrow $\freccia{\qot{A}{\rho}}{V_P(f)}{\qot{B}{\sigma}}$ defined as the vertical arrow $a$ of the following diagram
\[
\xymatrix@+2pc{
&& \qot{A}{\rho}\ar[d]^a\\
A \ar[rru]^{q_{\rho}}\ar[r]_f &B \ar[r]_{q_{\sigma}} & \qot{B}{\sigma}.
}\]
The arrow $V_P(f)=a$ exists by the universal property of quotients, because $\rho\leq P_{f\times f}(\sigma)\leq P_{f\times f} P_{q_{\sigma}\times q_{\sigma}}(\delta_{\qot{B}{\sigma}})$. The natural transformation $v_P$ is defined by the following components: for every object $(A,\rho)$ of $\mR_P$, we have that $\freccia{P_q(A,\rho)}{(v_P)_{(A,\rho)}}{P(\qot{A}{\rho})}$ acts as $\alpha\mapsto (P_{q_{\rho}})^{-1}(\alpha)$. Notice that $\freccia{P(\qot{A}{\rho})}{P_{q_{\rho}}}{\des_{\rho}}$ is invertible because $P$ is a doctrine of $\QED$, and then quotients are effective and also effective descent, i.e. $P(\qot{A}{\rho})\cong \des_{\rho}$. 
\begin{lemma}
With the previous assignments $\freccia{P_q}{(V_P,v_P)}{P}$ is a 1-cell of $\QED$.
\end{lemma}
\begin{proof}
First we show that $\freccia{\mR_P}{V_P}{\mC}$ is a functor. Let us consider the arrows $\frecciasopra{(A,\rho)}{f}{(B,\sigma)}$ and $\frecciasopra{(B,\sigma)}{g}{(C,\gamma)}$ of $\mR_P$. Now we show that $V_P(gf)=V_P(g)V_P(f)$. Observe that the diagram
\[\xymatrix@C=2pc{
A\ar[rrr]^{q_{\rho}}\ar[dr]_f &&& \qot{A}{\rho}\ar[d]^{V_P(f)}\\
& B\ar[rr]^{q_{\sigma}}\ar[dr]_g&&\qot{B}{\sigma}\ar[dd]^{V_P(g)}\\
&&C\ar[rd]_{q_{\gamma}}\\
&&&\qot{C}{\gamma}
}\]
commutes, and then we have that $V_P(g)V_P(f)$ is the unique arrow such that $V_P(g)V_P(f) q_{\rho}= q_{\gamma}gf$, hence it is exactly $V_P(gf)$. Moreover $V_P(\id)=\id $, and then $\freccia{\mR_P}{V_P}{\mC}$ is a functor, and it is direct to check that it preserves finite products. Therefore, we can conclude that $(V_P,v_P)$ is a 1-cell of $\ED$, because  $V_P$ is a preserving product functor, and $v_P$ is a natural transformation, whose naturality follows from the fact that every components $(v_P)_{(A,\rho)}$ is an iso, and for every 1-cell $(F,b)$, we have that $\ovln{F}(f)$ is $F(f)$ and $\ovln{b}$ acts as $b$.

One can show that $(V_P,v_P)$ is also a morphism of $\QED$, i.e. $V_P$ preserves quotients, by using the same idea of Lemma \ref{prop C is well def on 2-cell}. Since every $P_q$-equivalence relation $\tau$ on $(A,\rho)$ is also a $P$ equivalence relation on $A$, it is easy to see that the $P_q$-quotient of $\tau$ is $\freccia{(A,\rho)}{q_{\tau}=[id_A]}{(A,\tau)}$. Hence $V_P(q_{\tau})$ is the unique arrow
\[\xymatrix@+2pc{
A\ar[r]^{q_{\rho}}\ar[rd]_{q_{\tau}} &\qot{A}{\rho}\ar[d]^{V_P(q_{\tau})}\\
& \qot{A}{\tau}
} \]
which is exactly a quotient map of $\tau\in P(\qot{A}{\rho}\times \qot{A}{\rho})$. Hence we have proved that $(V_P,v_P)$ is a 1-cell of $\QED$.
\end{proof}
\begin{theorem}\label{theorem quotient completion}
The 2-functor $\freccia{\ED}{\funQ}{\QED}$ is 2-left adjoint to the forgetful functor $\freccia{\QED}{\funQ}{\ED}$. The unit of this 2-adjunction $\freccia{\id_{\ED}}{\eta}{\funU\funQ}$ is given by $\eta_P=(L_P,l_P)$ the counit $\freccia{\funQ\funU}{\varepsilon}{\id_{\QED}}$ is given by $\varepsilon_P=(V_P,v_P)$.
\end{theorem}

As in Sections \ref{sec doctrine with comp} and \ref{section comprensive diag} consider the following 2-monad, given 2-adjunction of Theorem  \ref{theorem quotient completion}: this theorem induces a 2-monad $\freccia{\PD}{\mT_q}{\PD}$, whose unit is given by the unit of the 2-adjunction, and whose multiplication is defined by $\mu=\varepsilon \funQ$, as in 1-dimensional case.
Again, we study that 2-monad $\mT_q$ and we show that we have the equivalence of 2-categories
$$\alg{\mT_q}\equiv \QED .$$

We start by showing that every elementary doctrine of $\QED$ is a $\mT_q$-algebra . 
\begin{proposition}\label{proposition P has quotient imples T-alg
}
Let $\doctrine{\mC}{P}$ be an elementary doctrine of $\CompAdj$, then $(P,\varepsilon_P)$ is a $\gromonad$-algebra.
\end{proposition}
\begin{proof}
The diagram
\[\comsquare{\mT_q^2 P}{\mT_q P}{\mT_q P}{P}{\mT_q \varepsilon_P}{\mu_P}{\varepsilon_P}{\varepsilon_P}\]
commutes because $\mu_P=\varepsilon_{\funQ P}$ and $\freccia{\mT_q}{\varepsilon}{\id_{\QED}}$ is a 2-natural transformation. Similarly we have that the unit axiom for strict algebras is satisfied.
\end{proof}

\begin{proposition}\label{proposition P algebras imples comprehensios}
Let $(P,(F,b))$ be a $\alg{\mT_q}$. Then the doctrine $P$ is a doctrine of $\QED$. Moreover $(F,b)=\varepsilon_P$.
\end{proposition}
\begin{proof}
It is direct to check that given an equivalence relation $\rho\in P(A\times A)$, then $\freccia{F(A,\rho)}{F(\id_A)}{F(A,\delta_A)}$ is a quotient of $\rho$, where $\freccia{(A,\rho)}{\id_A}{(A,\delta_A)}$ is the quotient of $\rho\in P_q(A,\delta_A)$. Moreover, it is direct to show that quotients are stable, and effective. Now we show that $(F,b)=\varepsilon_P$.
By the unit axiom of algebras, we have that $(F,b)\eta_P=id_P$, but since $P$ is a doctrine of $\QED$, we also have $\varepsilon_P \eta_P=\id_P$. Therefore we have that $(F,b)\eta_P=\varepsilon_P\eta_P$ implies that $(F,b)=\varepsilon_P$, because $F(A,\rho)=F(\qot{(A,\delta)}{\rho})=\qot{(F(A,\delta_A))}{b{(\rho)}}=\qot{(\varepsilon_P\eta_P(A))}{\rho}=\varepsilon_P(A,\rho)$. Similarly one can prove that $F(f)=\varepsilon_P(f)$, because every arrow $\freccia{(A,\rho)}{f}{(B,\sigma)}$ is the unique arrow such that the following diagram commutes 
\[\xymatrix@+2PC{
(A,\delta_A)\ar[r]^{q_{\rho}}\ar[d]_f & (A,\rho)\ar[d]^{f}\\
(B,\delta_B) \ar[r]_{q_{\sigma}} & (B,\sigma).
}\]
 because $\rho\leq P_{f\times f}(\sigma)$. Since both $F$ and $\varepsilon_P$ preserve quotients, and since $F(q_{\rho})=\varepsilon_P(q_{\rho})$ and $F(\frecciasopra{(A,\delta_A)}{f}{(B,\delta_B)})=\varepsilon_P(\frecciasopra{(A,\delta_A)}{f}{(B,\delta_B)})$, then $F(\frecciasopra{(A,\rho)}{f}{(B,\sigma)})$ must be equal to the arrow $\varepsilon_P(\frecciasopra{(A,\rho)}{f}{(B,\sigma)})$ (by the unicity of the mediating arrow in the universal property of quotients). Hence $F=\varepsilon_P$. Finally it is direct to check that $b=j_P$.
\end{proof}

\begin{proposition}\label{proposition every morf extend uniquile to T-morf}
Let $P$ and $R$ be two doctrines of $\QED$, and let $\freccia{P}{(F,b)}{R}$ be a 1-cell of $\PD$, then there exists a unique 2-cell $\omega$ such that $\freccia{(P,\varepsilon_P)}{((F,b),\tau)}{(R,\varepsilon_R)}$ is a lax morphism of $\alg{\mT_q}$.
\end{proposition}
\begin{proof}
Consider the square
\[ \comsquare{P_q}{R_q}{P}{R.}{\mT_q (F,b)}{\varepsilon_P}{\varepsilon_R}{(F,b)}\]
Let $(A,\rho)$ be an object of $\mR_P$. Then we have that \[\varepsilon_R \mT_q(F,b)(A,\rho)=\qot{(FA)}{b(\rho)}\]
where, to simplify the notation, we denote $b(\rho)=R_{\angbr{F\pr_1}{F\pr_2}^{-1}} b_{A\times A}(\rho)$, and
\[(F,b)\varepsilon_P (A,\rho)=F(\qot{A}{\rho}).\]
We define $\omega_{(A,\rho)}$ as the morphism
\[ \xymatrix@+2pc{
F(A) \ar[dr]_{F(q_{\rho})} \ar[r]^{q_{b(\rho)}} &\qot{(FA)}{b(\rho)}\ar[d]^{\omega_{(A,\rho)}} \\
& F(\qot{A}{\rho}) 
}\]
which exists by the universal property of quotients, because
\[ R_{F(q_{\rho})\times F(q_{\rho}) }(\delta_{F(\qot{A}{\rho}) })= R_{\angbr{F\pr_1}{F\pr_2}^{-1}}b_{A\times A}P_{q_{\rho}\times q_{\rho}}(\delta_{\qot{A}{\rho}})\leq R_{\angbr{F\pr_1}{F\pr_2}^{-1}}b_{A\times A}(\rho)=b(\rho).\]
Now we show that the $\omega$  is a natural transformation $\duefreccia{V_R\ovln{F}}{\omega}{FV_P}$. Let us consider an arrow $\freccia{(A,\rho)}{f}{(B,\sigma)}$ of the category $\mR_P$. Then the diagram 
\[\xymatrix@+2pc{
F(A)\ar[dd]_{V_R\ovln{F}(f)}\ar[rr]^{q_{b(\rho)}} \ar[rd]_{F(q_{b(\rho)})}& &F(A)_{b(\rho)}\ar[dd]^{V_R\ovln{F}(f)}\\
& (F\qot{A}{\rho})\ar[dd]_(0.3){FV_P(f)}\ar@{<-}[ru]_{\omega_{(A,\rho)}}\\
F(B)\ar[dr]_{F(q_{b(\rho)})}\ar[rr]|!{[d];[r]}\hole_(0.55){\;\;\;\;q_{b(\sigma)}} & &F(B)_{b(\sigma)}\\
& (F\qot{B}{\sigma})\ar@{<-}[ru]_{\omega_{(B,\sigma)}}.
}\]
commutes, because every triangle commutes and the left and the back squares commute, hence, using the fact that quotients are epi, we can show that the right square commutes. Therefore, $\omega$ is a natural transformation. Moreover using the same argument of Proposition \ref{proposition every morf extend uniquile to T-morf} we can conclude that $\omega$ is a 2-cell of $\QED$.

It is direct to show that $((F,b), \omega)$ satisfies the coherence axioms of lax morphisms of algebras. Again, following the idea of Proposition \ref{proposition every morf extend uniquile to T-morf}, we have that the following axiom is satisfied 
\[\xymatrix@+1pc{
P \ar[d]_{\eta_P}\ar[r]^{(F,b)} &R\ar[d]^{\eta_R}&&P\ar[dd]_{1_P}\ar[r]^{(F,b)}& R\ar[dd]^{1_B}\\
P_q \ar[d]_{\varepsilon_P} \ar[r]^{(\ovln{F},\ovln{b})} \xtwocell[r]{}<>{_<5>\omega}&R_q \ar[d]^{\varepsilon_R}\ar@{}[rr]|{=}&& \\
P \ar[r]_{(F,b)} & R&& P\ar[r]_{(F,b)} &R
}\]
because, when $\rho=\delta_A$, then we have that $\omega_{(A,\delta)}=\id_{FA}$.
Now we show that this $\omega$ is unique. Let us consider another 2-cell $\duefreccia{\mT_q(F,b)\varepsilon_R}{\theta}{\varepsilon_P(F,b)}$ such that $((F,b),\theta)$ is a lax-morphism 
\[\xymatrix@+1pc{
P_q \ar[r]^{(\wht{F},\wht{b})} \ar[d]_{\varepsilon_P} \xtwocell[r]{}<>{_<5>\;\;\omega
}& R_q \ar[d]^{\varepsilon_R}\\
P \ar[r]_{(F,b)} &R
}\]
of $\mT_q$ algebras. Then it must satisfy the following condition
\[\xymatrix@+1pc{
P \ar[d]_{\eta_A}\ar[r]^{(F,b)} &R\ar[d]^{\eta_B}&&P\ar[dd]_{1_P}\ar[r]^{(F,b)}& R\ar[dd]^{1_B}\\
P_q \ar[d]_{\varepsilon_P} \ar[r]^{(\ltil{F},b)} \xtwocell[r]{}<>{_<5>\theta}&R_q \ar[d]^{\varepsilon_R}\ar@{}[rr]|{=}&& \\
P \ar[r]_{(F,b)} & R&& P\ar[r]_{(F,b)} &R
}\]
and this means that $\theta_{(A,\delta_A)}=\id_{FA}$. Therefore, since $\theta$ is a natural transformation from $V_R\ovln{F}$ to $FV_P$, then the following diagram
\[\xymatrix@+2pc{
F(A)\ar[d]_{\theta_{(A,\delta_A)}} \ar[r]^{q_{b(\rho)}}&\qot{F(A)}{b(\rho)}\ar[d]^{\theta_{(A,\rho)}}\\
(FA) \ar[r]_{F(q_{\rho})}& F(\qot{A}{\rho})
}\]
commutes, and since $\theta_{(A,\rho_A)}=\id_{FA}$, then we have that $\theta_{(A,\rho)}$ must be $\omega_{(A,\rho)}$ because, by definition, $\omega_{(A,\rho)}$ is the unique arrow such that the diagram 
\[ \xymatrix@+2pc{
F(A) \ar[dr]_{F(q_{\rho})} \ar[r]^{q_{b(\rho)}} &\qot{(FA)}{b(\rho)}\ar[d]^{\omega_{(A,\rho)}} \\
& F(\qot{A}{\rho}) 
}\]
commutes. Hence $\theta=\omega$.
\end{proof}
Combining previous results we can prove the following theorem, showing the 2-monadicity of the elementary quotient completion.
\begin{theorem}
The 2-category $\alg{\mT_q}$ of strict algebras, algebras morphisms and $\mT_q$-transformation is 2-equivalent to the 2-category $\QED$.
\end{theorem}
Finally, we can conclude, again by directly applying the previous results, with the following result.
\begin{theorem}
The 2-monad $\freccia{\PD}{\mT_q}{\PD}$ is lax-idempotent.
\end{theorem}
So, as in the cases of the comprehension completion and the comprehensive diagonal completion, we have that having quotients is a property of an elementary doctrine.
\begin{remark}
Observe that, as in the case of comprehension completion, the results of this section can be generalized in the case we do not assume any choice of quotients. Again, in this case, the counit of the elementary quotient completion becomes a pseudo-natural transformation, and then the monad $\mT_q$ is just a pseudo-monad, and not a 2-monad. 
\end{remark}

\section{Distributive laws}\label{section pseudo-distributive laws}
In the previous sections we provided an algebraic framework to deal with the quotient completion and the comprehension completion of elementary doctrines. Taking the advantage of this presentation, we show how these free constructions interact, i.e. we provide a distributive law between the 2-monads $\mT_c$ and $\mT_q$. 

Recall from Section \ref{section 2dim monad}, in particular Theorem \ref{theorem ps-dist-law equivalent to lifting}, that showing the existence of a distributive law between the 2-monads $\mT_q$ and $\mT_c$ is equivalent to provide a lifting of $\mT_q$ on the 2-category $\alg{\mT_c}$. 

Notice that combining the equivalence of 2-categories $\alg{\mT_c}\equiv \CE$ we provide in Theorem \ref{theorem comp. comp. is monadic}, together with the result \cite[Lemma 5.3]{QCFF}, which states that if $P$ has comprehensions then the quotient completion $P_q$ has comprehensions as well, the construction of a lifting of $\mT_q$ on the 2-category $\alg{\mT_c}$ is quite direct.

\begin{lemma}\label{prop T_cPR is a functor}
The assignment 
\[\freccia{\alg{\mT_c}((P,a),(R,c))}{\ltil{\mT_q}_{(P,a)(R,c)}}{\alg{\mT_c}((P_q,\varepsilon_{P_q}),(R_q,\varepsilon_{R_q}))}\] 
mapping a 1-cell $(F,b)\mapsto \mT_q(F,b)$ and 
 and a 2-cell $(F,b)\theta\mapsto \mT_q\theta$
is a functor.
\end{lemma}
\begin{proof}
Since $(P,a)$ is a $\mT_c$-algebra then, by Theorem \ref{theorem comp. comp. is monadic}, we have that the doctrine $P$ has comprehensions, and by \cite[Lemma 5.3]{QCFF}, we can conclude that $P_q$ has comprehensions as well.
Similarly, one can directly check that if $(F,b)$ is a 1-cell of $\CE$ and $\theta$ is a 2-cell then $\mT_q(F,b)$ is again a 1-cell of $\CE$ and $\mT_q\theta$ is a 2-cell of $\CE$. Therefore we conclude that $\ltil{\mT_q}_{(P,a)(R,c)}$ is a functor.
\end{proof}
\begin{lemma}\label{prop the lifting is a 2-fun}
The functor defined in \ref{prop T_cPR is a functor} can be extended to a 2-functor 
\[ \freccia{\alg{\mT_c}}{\ltil{\mT_q}}{\alg{\mT_c}}\]
where $\ltil{\mT_q}(P,a):=(P_q, \varepsilon_{P_q})$. Moreover it is a 2-monad, whose identity and multiplication are those induced by $\mT_q$.
\end{lemma}
 

\begin{theorem}\label{theorem dist law}
There exists a distributive law $\freccia{\mT_c\mT_q}{\delta}{\mT_q\mT_c}$.
\end{theorem}
\begin{proof}
If we consider the forgetful 2-functor $\freccia{\alg{\mT_c}}{\funU_{\mT_c}}{\ED}$, we have the equality $\mT_q\funU_{\mT_c}=\funU_{\mT_c}\ltil{\mT_q}$. Then  $\ltil{\mT_q}$ is a lifting of $\mT_q$. Thus, by applying Theorem \cite[Theorem 1]{PDLAVB} to conclude that there exists a distributive law $\freccia{\mT_c\mT_q}{\delta}{\mT_q\mT_c}$.
\end{proof}
\begin{corollary} 
The 2-functor $\mT_q\mT_c$ is a 2-monad.
\end{corollary}
\begin{proof}
It follows by Theorem \ref{theorem the composite ps-monad is a ps monad} and Theorem \ref{theorem dist law}. 
\end{proof}

\begin{remark}
As observed in \cite{EQC}, the 2-monad $\mT_d$ which freely adds comprehensive diagonals in general does not preserve full comprehensions and quotients. In particular, if $P$ has full comprehensions, then $\mT_d(P)$ has only weak comprehensions, and similarly, if $P$ effective quotients, the doctrine $\mT_d(P)$ has only a weak form of quotients. Therefore this 2-monad $\mT_d$ cannot be lifted to the 2-categories of $\alg{\mT_c}$ and $\alg{\mT_q}$.
\end{remark}

\begin{remark}
Recall from \cite{QCFF,EQC} that all the 2-monads $\mT_c$, $\mT_q$ and $\mT_d$ preserve the existential structure of a doctrine, i.e. if $P$ is an elementary existential then all the doctrines $\mT_c(P)$ $\mT_q(P)$ and $\mT_d(P)$ are elementary and existential. Therefore, if we consider the 2-monad $\mT_e$ of the existential completion introduced in \cite{ECRT}, we have that all these 2-monad can be lifted to 2-monads on $\alg{\mT_e}$, i.e. the 2-category of elementary and existential doctrines. Hence we can conclude that there are the following distributive laws:
\begin{itemize}
\item $\freccia{\mT_e\mT_c}{\delta_1}{\mT_c\mT_e}$;
\item $\freccia{\mT_e\mT_q}{\delta_2}{\mT_q\mT_e}$;
\item $\freccia{\mT_c\mT_d}{\delta_3}{\mT_d\mT_e}$.
\end{itemize}
Again this is given by the correspondence between lifting of 2-monads and distributive laws.
\end{remark}

\section*{Acknowledgements}
My acknowledgements go first my supervisor Pino Rosolini, for his indispensable comments and
suggestions. Then, I wish also to thank Milly Maietti, Fabio Pasquali and Jacopo Emmenegger for useful discussions and comments on the elementary quotient completions.
\bibliographystyle{apalike}
\bibliography{biblio_tesi_PHD}

\end{document}